\documentclass[11pt]{article}
\usepackage{stmaryrd} %Is it useful really? Lbag is in here but I
\usepackage{latexsym}
\usepackage{amsmath,amsthm,amsfonts,amssymb, mathrsfs, wasysym} %way better.

\renewcommand{\topfraction}{.85}
\renewcommand{\bottomfraction}{.7}
\renewcommand{\textfraction}{.15}
\renewcommand{\floatpagefraction}{.66}
\renewcommand{\dbltopfraction}{.66}
\renewcommand{\dblfloatpagefraction}{.66}
\newtheorem{thm}{Theorem}[section]
\newtheorem{theo}[thm]{Theorem}

\newtheorem{defi}[thm]{Definition} 

\newtheorem{coro}[thm]{Corollary}

\newtheorem{prop}[thm]{Proposition}

\newtheorem{lemma}[thm]{Lemma}

\newtheorem{rem}[thm]{Remark}
\newtheorem{remark}[thm]{Remark}

\makeatletter
\@tfor\next:=abcdefghijklmnopqrstuvwxyzABCDEFGHIJKLMNOPQRSTUVWXYZ\do{%
  \def\command@factory#1{%
    \expandafter\def\csname cal#1\endcsname{\mathcal{#1}}
    \expandafter\def\csname frak#1\endcsname{\mathfrak{#1}}
    \expandafter\def\csname scr#1\endcsname{\mathscr{#1}}
    \expandafter\def\csname bb#1\endcsname{\mathbb{#1}}
    \expandafter\def\csname rm#1\endcsname{\mathrm{#1}}
  }
 \expandafter\command@factory\next
}
\makeatother

\newcommand*{\longhookrightarrow}{\ensuremath{\lhook\joinrel\relbar\joinrel\rightarrow}}
\newcommand {\tto}{ \to\rangle }
\newcommand {\onto} {\twoheadrightarrow}
\newcommand {\into} {\hookrightarrow}
\newcommand {\xra} {\xrightarrow}    
\newcommand{\ra}{\rightarrow}
\newcommand{\imp} {\Rightarrow}
\newcommand{\actedon}{\curvearrowleft} 
\newcommand{\actson}{\curvearrowright}

\newcommand {\sd} {\rtimes}   
\newcommand{\semidirect}{\ltimes}
\newcommand{\isemidirect}{\rtimes}
\newcommand{\tensor}{\otimes}

\newcommand{\Dunion}{\bigsqcup} 
\newcommand{\disjoint}{\sqcup}
\newcommand{\normal} {\vartriangleleft}

\newcommand {\ie}{ i.e.  }

\newcommand{\ul}[1]{\underline{#1}} 
\newcommand{\ol}[1]{\overline{#1}}

\newcommand{\Fix} {\operatorname{Fix}}

\newcommand{\diam}{\mathop{\mathrm{diam}\;}}
\newcommand{\Zmax}{Z_{max}}
\newcommand{\ZmJSJ}{$\mathcal{Z}_{{\rm max}}$-JSJ }
\newcommand{\ZJSJ}{$\mathcal{Z}$-JSJ }
\newcommand{\ad}{{\rm ad}}
\newcommand{\Aut}{{\rm Aut}}
\newcommand{\Out}{{\rm Out}}
\newcommand {\Inn} {{\rm Inn}}
\title{On suspensions, and conjugacy of  hyperbolic  automorphisms
  {\it and} of a few more.}
\author{Fran\c{c}ois Dahmani\thanks{Partially supported by the ANR
    2011-BS01-013 and the Institut Universitaire de France}}
\begin{document}

\maketitle

\begin{abstract} The two parts of this work have been through
  different publishing processes. The main reason of this is that part
  2, and only part 2, relies  at the time of writing on unpublished
  material.   The two parts  are shown here in a single
  arXiv document because it did not make much sense to separate them
  for other purposes than the one explained above. Please notice
  that numbering has been chosen to match the published
  versions. Therefore there are two sections called 1, etc.  There are
  also    inherently some repetitions ({\it e.g.} in both introductions), I
  apologise for that. 

  The first part (page \pageref{part1}) proposes a solution to the conjugacy problem for
  hyperbolic automorphisms of finitely presented groups. 

  The second part (page \pageref{part2}) proposes a solution  to the conjugacy problem for  {\it some} relatively hyperbolic automorphisms of free groups. 
\end{abstract}
\newpage

\tableofcontents

\newpage

\part{On suspensions, and conjugacy of hyperbolic automorphisms.} \label{part1}

\begin{abstract} We remark that the conjugacy problem for pairs of hyperbolic automorphisms of
  a finitely presented group (typically a free group) is decidable. The solution that we propose
  uses the isomorphism problem for the suspensions, and the study of
  their automorphism group. 
  \end{abstract}

\subsection*{Introduction}

Let $F$ be a finitely presented group, $\Aut(F)$ be its automorphism
group,  and $\Out(F) = \Aut(F)/\Inn (F)$ be its outer automorphism group.

A way to consider the conjugacy problem in $\Aut(F)$, or
$\Out(F)$, is to relate it to an isomorphism problem on semi-direct
products of $F$ with $\mathbb{Z}$.

 Given two semi-direct products, $F\sd_\alpha \langle t \rangle$  and
 $F\sd_\beta \langle t' \rangle$, their structural automorphisms
 $\alpha$ and $\beta$ are conjugated in $\Aut(F)$ if and only if there
 is an isomorphism $F\sd_\alpha \langle t \rangle   \to  F\sd_\beta
 \langle t' \rangle$ sending  $F$ on $F$, and $t$ on $t'$. They are
 conjugated in  $\Out(F)$  if and only if there is an isomorphism
 sending $F$ on $F$, and $t$ in $t'F$ (though it is a well known fact
 with a standard proof,  we refer to Lemma \ref{lem;wellknown} for the
 version that we'll use).

By analogy with topology and dynamical systems, we wish to call such semi-direct products
{\it suspensions} of $F$, in which $F$ is the {\it fiber} and $t$ is
the choice of a {\it transverse direction}, and $tF$ is the choice of
a transverse {\it orientation}. The conjugacy problem in $Out(F)$ can be expressed as the problem of determining whether
suspensions are fiber-and-orientation-preserving isomorphic.

 We carry out this approach for automorphisms of finitely
presented groups producing word-hyperbolic suspensions.

Consider for instance $F=F_n$ a free group of finite rank $n$. 
 In that case, a solution to the conjugacy problem in $\Out(F_n)$ was announced
by Lustig \cite{Lu1, Lu2}. However, it might still be desirable to find
short\footnote{In the sense that the exposition is 
  short; in this paper  we ostensibly ignore algorithmic complexity, and to some extend conceptual complexity hidden in the tools that are used.} complete solutions for specific classes of elements in
$\Out(F_n)$. For instance, consider the class of  atoroidal
automorphisms: those whose powers   do not preserve any conjugacy class
beside $\{1\}$. Since  Brinkmann proved in \cite{Br_gafa}  that an
automorphism produces a hyperbolic suspension if and only if
it is atoroidal, there is a conceptually simple (slightly brutal) way  to algorithmically check
whether a given automorphism is indeed atoroidal. It consists in
two parallel procedures. The first one looks for a
preserved conjugacy class, by enumerating elements, and their images
by powers of the given automorphism; it halts when a non trivial element
is found such that its image by a non-trivial power of the
automorphism are found to be conjugated in $F_n$. The second  is
Papasoglu's procedure \cite{Pap} applied on the semi-direct product
of $F_n$ by $\mathbb{Z}$ (with given structural automorphism), that
halts if and only if the semi-direct product is
word-hyperbolic. Brinkmann's results says that exactly one of these two
procedures will halt, and depending which one halts, we deduce whether
or not the automorphism is atoroidal.

If two given automorphisms are found to be atoroidal with this procedure, our
main result will allow to decide whether they are conjugate in $\Out F_n$.

 For hyperbolic groups, the isomorphism problem is solved \cite{Sel, DGr,
  DGu_gafa}. In several examples, the solution available can settle the
conjugacy problem. Take two pseudo-Anosov diffeomorphisms of a
hyperbolic surface. The mapping tori are closed hyperbolic
$3$-manifold, hence hyperbolic and rigid (in the sense that their
outer-automorphism groups are finite). Sela's solution to the
isomorphism problem of their fundamental groups \cite[0.3,  7.3]{Sel} 
 provides all conjugacy
classes of isomorphisms (there are finitely many), and from that
point, it is possible to check whether one of them preserves
the fiber and the orientation. 

For automorphisms of a free group, the analogous situation is when the two
automorphisms are atoroidal, fully irreducible (with irreducible
powers), and for their conjugacy problem, see \cite{Los, Lu3, Sel},
the later (in Coro. 0.6 {\it loc. cit.}),  using the same strategy as above. However, there are atoroidal automorphisms for which the
suspension, though hyperbolic, is not rigid. In \cite{Br_split}
Brinkmann gave several examples with different behaviors. In
particular, the solution to the isomorphism problem of hyperbolic
groups will not reveal all isomorphisms between suspensions, and since
the fibers are exponentially distorted in the suspensions,  the usual rational tools
(see \cite{D_israel, DGu_JoT}) do not work for solving the isomorphism problem
with such a preservation constraint.
 One can thus merely detect the
existence of one isomorphism (say $\iota$), but for investigating the existence of
an isomorphism with the aforementioned properties, 
one is led to consider an orbit
problem of the automorphism group of $F\sd \langle t \rangle$: decide whether
an automorphism sends $\iota(F)$ on $F$ and $\iota(t)$    
 in $t'F$.  

Orbits problems are not necessarily easier, especially if the group
acting is large and complicated. In \cite{BMV}, for instance,
Bogopolski, Martino and Ventura propose a subgroup of
$GL(4,\mathbb{Z})$ whose orbit problem on $\mathbb{Z}^4$ is undecidable.

In 
this paper we prove that, if $F$ is finitely generated  and  $F\sd \langle t \rangle$
hyperbolic, then   $\Out(F\sd \langle t \rangle)$  contains a finite
index abelian subgroup,  whose action on $H_1(F\sd \langle t \rangle)$
is generated by transvections. This allows us to prove that the specific orbit problem
above is solvable in that case, by reducing it to a system of linear Diophantine
equations, read in  $H_1(F\sd \langle t \rangle)$. This is explained
in \ref{sec;orbit_Mod}.

These are thus the key steps to produce what we see as a picturesque way of solving the
conjugacy problem for automorphisms of  finitely presented groups with
hyperbolic suspension (Theorem \ref{thm;decide_fop_iso_for_hyp}). 

The proof that  $\Out(F\sd \langle t \rangle )$   is virtually
abelian is the conjunction of Proposition \ref{prop;Mod_ab} (together with the remark \ref{rem;Mod_ab}, which
is not actually needed in the rest of the proof)  and Proposition
\ref{prop;coset_rep}. The proof of the later is is done by considering the canonical JSJ decomposition of the
hyperbolic group $F\sd \langle t\rangle$, and by proving that  
this graph-of-groups
decomposition does not contain any surface vertex group. We cannot resist to 
 sketch the proof of this key fact (that will be detailed in \ref{sec;no_surf}, and that takes
roots in the way Brinkmann produces his examples in \cite{Br_split}). 
 Consider the tree of the JSJ decomposition $T$, and
$\mathbb{X}$ the graph of group quotient of $T$ by $G = F\sd
\langle t \rangle$. Since $F$ is normal in $G$, $T$ is a minimal tree for
$F$, and $\mathbb{Y}= F\backslash T$ is a graph-of-groups decomposition of $F$ whose underlying graph  is its own core, and since its genus is
bounded by the rank of $F$, it is finite. 
It follows 
that every vertex group (resp. edge group) in $\mathbb{X}$ is the suspension of a vertex group  (resp. edge group) in $\mathbb{Y}$:  lift the vertex in
$T$,  where  its $\langle t\rangle $-orbit passes twice on a pre-image
of a vertex in $\mathbb{Y}$, thus yielding the suspension (see Lemma
\ref{lem;vert_are_susp} for more details).  Since edge groups in
$\mathbb{X}$ are cyclic,  edge groups in $\mathbb{Y}$
must be trivial.  Therefore $\mathbb{Y}$ is a free decomposition of $F$,
and its vertex groups are of finite type. Going back to $\mathbb{X}$
again,   vertex groups
of $\mathbb{X}$ are suspensions of infinite, finitely generated groups, hence cannot be free nor surface groups, because 
finitely generated normal subgroups of free groups (or surface groups) are of
finite index or trivial.

To present these 
arguments,  the formalism of automorphisms of
graph of groups is to be recalled,  in a rather precise way in
order to be useful. The confident reader may skip this part (section \ref{sec;prelim}), and only
retain that the small modular group is the group of automorphisms
generated by Dehn twists on edges of a graph of groups.

I am grateful to the referee for useful comments.

\section{Preliminary on automorphisms of graphs-of-groups}\label{sec;prelim}

\subsection{Trees and splittings}

We fix our formalism for graphs and graph of groups. This material is
classical, and can be found in Serre's book \cite{Serre}. A graph is a tuple $X=(V, E, \bar{ },
o, t)$
where $V$ is a set (of vertices), $E$ is  a set (of oriented edges)
and $\bar{ } :E\to E$, $o:E\to V$, $t:E\to V$ verify $(\bar{ } \circ \bar{
  }) = Id_E$ and $(o\circ \bar{ }) = t$.
A graph-of-groups $\mathbb{X}$  consists of a graph $X$, for each vertex $v$ of $X$,  a group $\Gamma_v$, for each unoriented edge $\{e,\bar e\}$ of $X$, a group $\Gamma_{\{e, \bar e\}}$ (but we will write $\Gamma_e= \Gamma_{\bar e}$ for it), and for each oriented edge $e$ of $X$, a injective homomorphism $i_e:\Gamma_e \to \Gamma_{o(e)}$, where $o(e)$ is the origin vertex of the oriented edge $e$.

The Bass group $B(\bbX)$ is $$B(\bbX)= \langle (\bigcup_{v\in V}
\Gamma_v) \cup E\; |  \; \forall e \in E, \forall g\in \Gamma_e, \;
\bar e = e^{-1}, \,   e^{-1}  i_{\bar e} (g)  e = i_e(g) \rangle.$$ 

An element $g_0e_1g_1e_2\dots g_ne_ng_{n+1}$ is a {\it path element} from $o(e_1)$ to $t(e_n)$  if $g_{i-1}\in o(e_i)$ and $g_i \in t(e_i)$ for all $i$.

The {\it fundamental group} $\pi_1(\bbX, v_0)$ of the graph-of-groups
$\bbX$ {\it at the vertex}  $v_0$,  is the subgroup of the Bass group
consisting of path elements from $v_0$ to $v_0$.

Choose a maximal subtree $\tau$ of the graph $X$, and consider
$\pi_1(\bbX, \tau) = B(\bbX)/\langle\langle e \in \tau
\rangle\rangle $. Then the
quotient map $ B(\bbX) \to \pi_1(\bbX, \tau)$ is, in restriction to 
$\pi_1(\bbX, v_0)$, an isomorphism.

A {\it $G$-tree} is a simplicial tree with a simplicial action of $G$
without inversion. It is minimal if there is no proper invariant
subtree. It is reduced if the stabilizer of any edge is a proper
subgroup of the stabilizer of any adjacent vertex. 

The quotient of a $G$-tree by $G$ is naturally marked by the family of
conjugacy classes of stabilizers of vertices and edges, and inherits a
structure of a graph-of-groups  (whose
fundamental group is isomorphic to $G$). 

The {\it Bass-Serre tree} of a graph-of-groups is its universal covering in the sense of graphs of groups. It is a $\pi_1(\bbX, v_0)$-tree.

A {\it collapse} of a $G$-tree $T$ is a $G$-tree $S$ with an equivariant map from $T$ to $S$ that sends each edge on an edge or a vertex and no pair of edges of  $T$ in different orbits are sent in the same edge (but they may be sent on the same vertex). A collapse of an edge $e$ in a graph-of-groups decomposition is a collapse of the corresponding Bass-Serre tree in which only the edges in the orbit of a preimage of $e$ are mapped to a vertex.

Given a group $G$, and a class of groups $\calC$,  a {\it splitting} of $G$ over groups in $\calC$ is an isomorphism between $G$ and the fundamental group of a graph of groups whose edge groups are in  $\calC$. Equivalently, a splitting of $G$ can be though of as an action of $G$ on a tree, with edge stabilizers in $\calC$. A splitting is called reduced if, in the tree,  there is no edge whose stabilizer equals that of an adjacent vertex.

\subsection{Automorphisms, and the small modular group $Mod_{\bbX}$} \label{sec;autom_gog}

Let $\bbX$ and $\bbX'$ be two graphs-of-groups. An isomorphism of graphs-of-groups $\Phi:\bbX\to \bbX'$ is a tuple $(\Phi_X, (\phi_v), (\phi_e), (\gamma_e))$ such that 
\begin{itemize}
\item $\Phi_X : X \to X'$ is an isomorphism of the underlying graphs, 
\item for all vertex $v$, $\phi_v:\Gamma_v \to \Gamma_{\Phi_X(v)}$ is an isomorphism, for all edge $e$, $\phi_e=\phi_{\bar e} :\Gamma_e \to \Gamma_{\Phi_X(e)}$ is an isomorphism, 
\item and for each edge $e$, $\gamma_e \in  \Gamma_{\Phi_X(t(e))}$
  satisfying,  for $v=t(e)$, $$\phi_v\circ i_e = \ad_{\gamma_e}\circ
  i_{\Phi_X(e)} \circ \phi_{e}, $$ for  $(\ad_{\gamma_e} : x\mapsto
\gamma_e^{-1} x\gamma_e)$ 
 the inner automorphism of
$\Gamma_{\Phi_X(t(e))}$ defined by the conjugacy by $\gamma_e$.
\end{itemize}

The last point is the commutation of the following diagram (for each
edge $e$): 
\begin{equation}\label{diagram;Bass} \begin{array}{ccc} \Gamma_v &
    \stackrel{\phi_v}{\rightarrow} \Gamma_{\Phi_X(v)}  \stackrel{\ad_{\gamma_e} }{\leftarrow} &  \Gamma_{\Phi_X(v)}  \\    \uparrow &  &  \uparrow \\
\Gamma_e & \stackrel{\phi_e}{\longrightarrow}   & \Gamma_{\Phi_X(e)}  \end{array}\end{equation}

When  $\bbX =\bbX'$ the isomorphisms can be composed in a natural way (see
\cite[2.11]{Bass}: $(\Phi_X, (\phi_v), (\phi_e), (\gamma_e))  \circ
(\Psi_X, (\psi_v), (\psi_e), (\eta_e))  = (\Phi_X\circ \Psi_X,
(\phi_{\Psi_X(v)}\circ \psi_v), (\phi_{\Psi_X(e)}\circ \psi_e),
(\phi_{\Psi_X(t(e))} (\eta_e)\gamma_{\Psi_X(e)}))   $ ), thus
providing  the automorphism group of the graph-of-groups $\bbX$, denote by $\delta\Aut(\bbX)$.

 This group $\delta\Aut(\bbX)$ naturally maps
into the automorphism group of the Bass group $B(\bbX)$ in the
following way: for all edge $\epsilon$, and all automorphism $\Phi =
(\Phi_X, (\phi_v), (\phi_e), (\gamma_e)) \in
  \delta\Aut(\bbX)$,  one has $\Phi(\epsilon) =  \gamma_{\bar
    \epsilon}^{-1} \Phi_X(\epsilon)
  \gamma_\epsilon$, and $\Phi|_{\Gamma_v} = \phi_v$. 
One can check that, for any $\Phi$,  the relations of the Bass group are
  preserved, 
 and that the thus induced morphism is bijective. 

\begin{remark} For this argument, see \cite[2.1,
 2.2]{Bass}, but notice that Bass chose to let conjugations act on the
 left, while we chose to let them act on the right (as in
 \cite{DGu_gafa}). This difference yields a few harmless inversions in
 the formulae, (actually the attentive reader may have spotted them
 already in the relations of the Bass group, \cite[1.5 (1.2)
]{Bass}).
 and the only risk here is to mix both (incompatible) choices.
\end{remark}

Each automorphism in $\delta\Aut(\bbX)$ sends path elements to path
elements,  hence naturally provides an outer-automorphism of
$\pi_1 (\bbX, v_0)$ (and a genuine  automorphism if $\Phi_X(v_0) =
v_0$), and we will often implicitly make this identification.

Let us define the {\it small modular group} of $\bbX$, denoted by
$Mod_{\bbX}$, to be the subgroup of  $\delta\Aut(\bbX)$ consisting of
elements of the form $(Id_X, (\phi_v), (Id_{\Gamma_e}),
(\gamma_e))$, for $\phi_v\in \Inn \Gamma_v$ inner automorphisms. One can check, using the
composition rule,  that this
forms a subgroup of $\delta \Aut (\bbX)$.

If we note 
$\phi_v= \ad_{\gamma_v}$, then   the compatibility
condition (\ref{diagram;Bass})  imposes that $\gamma_v \gamma_e^{-1}  \in
Z_{\Gamma_{t(e )}}(i_e (\Gamma_e))$ for all edge $e$ (where
$Z_{\Gamma_{t(e )} }( i_e (\Gamma_e))$ denotes the centralizer of the attached edge
group $   i_e(\Gamma_e)   $ in the vertex group $\Gamma_{t(e )}$).

If one chooses a generating set $S_e$ for each group  $Z_{\Gamma_{t(e
    )}}(i_e (\Gamma_e))$, then the small modular group is generated by 
the union of two collections of elements. First, the collection of {\it oriented  Dehn twists}
 $\{D_{\epsilon ,\gamma}, \, \epsilon\in E, \gamma\in S_\epsilon\}$,
 defined by $D_{e ,\gamma} = (Id_X, (Id_v), (Id_e), (x_e))$ for $x_\epsilon =
 \gamma$ and $x_e = 1$ for all $e\neq \epsilon$. Second, the collection of {\it inert twists}, $\{  (Id_X, (\ad_{\gamma_v})),
 (Id_e), (\gamma_e =\gamma_{t(e)})  ) \}$ (note that it is the same
 family of elements involved as $(\gamma_e)$ and in defining the
 $(\phi_v)$, and that this defines an element of $Mod_{\bbX}$).

The inert twists are not so interesting. They correspond to changing
some choice of fundamental domain in the Bass-Serre tree. In
particular:

\begin{lemma}
 Any inert twist vanishes in $\Out(\pi_1(\bbX, v_0))$.
\end{lemma}
\begin{proof}
Consider an inert twist $\Phi$. After conjugation over the whole group, we may assume that $\phi_{v_0}
= Id_{\Gamma_{v_0}}$. Then take a path element in the Bass group,
$g_0e_1 \dots g_n e_n g_{n+1}$ that is a loop from $v_0$ to $v_0$. 
  Each $g_i$ is in $\Gamma_{t(e_i)} =
\Gamma_{t(\overline{e_{i+1}})}$.  Thus, $\Phi(g_i) = \ad_{\gamma_i}
(g_i)$, and $\Phi(e_i) =   \gamma_{i-1}^{-1} e_i \gamma_i
$. Concatenation makes everything collapse, and $\Phi(w) = w$.
\end{proof}

We record  how oriented Dehn twists 
are realized as automorphisms of $\pi_1 (\bbX, v_0)$.  The following
is an immediate consequence of the definitions.

\begin{lemma}\label{lem;whataDehntwistdoes} 
Let $\epsilon \in E$ be an oriented edge of $\bbX$, and $\gamma \in
Z_{\Gamma_{t(\epsilon)}}(i(\Gamma_\epsilon))$. 

Let $e$ be an oriented edge different from  $ \epsilon$ and $\bar \epsilon$.

 The oriented Dehn twist $D_{\epsilon,\gamma}$, seen as an
 automorphism of $B(\bbX)$, is such that 
$D_{\epsilon,\gamma} (\epsilon) =
 \epsilon\gamma$,   $D_{\epsilon,\gamma} (\bar \epsilon) =
 \gamma^{-1} \bar \epsilon$  and $D_{\epsilon,\gamma} (e) = e$. 

Moreover, for all vertex $v$ in $X$, and $\gamma_v \in \Gamma_v$,
$D_{\epsilon,\gamma} (\gamma_v)= \gamma_v$.

\end{lemma}

One should be nonetheless cautious with the interpretation of
$D_{\epsilon,\gamma}$ as an automorphism of $\pi_1(\bbX, v_0)$, since
the identification of a preferred copy of $\Gamma_v$ in $\pi_1(\bbX,
v_0)$ is subject to a choice of path from $v_0$ to $v$. If this path
contains a (or several) copy of the edge $\epsilon$, then  $D_{\epsilon,\gamma}$ is
actually conjugating, on the right,  $\Gamma_v$ by $\gamma^{-1}$ (or a power of it).

\section{The $\Aut(G)$-orbit of the fiber of a suspension}

\subsection{An orbit problem for  the small modular group} \label{sec;orbit_Mod}

In this section, we discuss an orbit problem for $Mod_{\bbX}$ in
$H_1(G)$, the abelianisation of $G$. For that, we note that
$Mod_{\bbX}$ naturally maps into $\Aut(H_1(G)) \simeq GL(H_1(G))$.  We
will denote by $\bar{ } \, : G\to H_1(G)$ the abelianisation map.  Let
us also choose $E^+ \subset E$ a set of representatives of unoriented edges.

\begin{prop}\label{prop;Mod_ab}

 Let $G= \pi_1( \bbX,v_0)$. 
 The image of  $Mod_{\bbX}$ in $GL(H_1(G))$ is abelian,
 generated by transvections.
\end{prop}

\begin{proof} Since inert twists vanish in $\Out(G)$,  the images of
  oriented Dehn twists, generate the image of  $Mod_{\bbX}$ in
  $GL(H_1(G)$. It suffice to show that oriented Dehn
  twists   induce
  transvections on $H_1(G)$ that commute.

Any element $\gamma$ of $\pi_1(\bbX,v_0)$ has an expression as a
normal form coming from the ambient Bass group: $\gamma =
g_0e_0g_1e_1\dots e_m g_{m+1}$ where for all $i$, $e_i \in E$
and $g_{i+1}\in \Gamma_{t(e_i)}, \, g_0, g_{m+1}\in \Gamma_{v_0}$.

The
normal form of $\gamma$ turns, in $H_1(G)$,  into $\bar{\gamma} = g_0\dots g_{m+1}
\times \prod_{e\in E^+ } e^{n(\gamma,e)}$ where $n(\gamma,e)$ is the number of occurrences of $e$ in $(e_0,\dots,e_m)$ minus the number of occurrences of $\bar e$. 

If $D_{\epsilon,h}$ is the Dehn twist of $h$ on $\epsilon$, the
induced automorphism on $H_1(G)$ is denoted by $\overline{D_{\epsilon,h}}$.
From the expression of $D_{\epsilon,h}(e)$ in
\ref{lem;whataDehntwistdoes},  
  $\overline{D_{\epsilon,h}}$ is the following transvection   
 $$ \overline{D_{\epsilon,h}} (\bar \gamma) = \left(  g_0  g_1
   \dots  g_{m+1}\right)  \times \left( \prod_{e\in E^+ }
   e^{n(\gamma, e)} \right)  \times \bar h^{n(\gamma,\epsilon)}
 \; = \; \bar \gamma \times \bar h^{n(\gamma,\epsilon)}. \qquad (*)
$$

Since $h\in \Gamma_{t(\epsilon)}    $, the element
$h^{n(\gamma,\epsilon)}$ is in a vertex group, and it is fixed (if
seen in the Bass group) or conjugated (if seen in $G$) by all
oriented Dehn twists on edges, thus oriented Dehn twists on edges commute in the abelianisation.
\end{proof}

\begin{rem}\label{rem;Mod_ab}
In fact, if  $Z_{\Gamma_{t(e)}}(i_e(\Gamma_e)  )$ is abelian
(which is the case if $G$ torsion free hyperbolic), then
$Mod_{\bbX}$  is abelian itself.
 \end{rem}

We keep the same notations. Note that the group $H_1(G)/ \bar F$ is infinite cyclic, by
assumption, and is  generated by the image of $\bar t$. Define
$\bar{\bar{ }}:G\to H_1(G)/\bar F$ and for all $\gamma \in G$, define
$\delta(\gamma)$ to be the unique integer such that $\bar{\bar{\gamma}} = \bar{\bar{t}}^{\delta(\gamma)}$.

\begin{prop}\label{prop;orbit_fiber}
  Let $G$ be a finitely generated group that can be expressed as a
  semi-direct product $F \sd \langle t\rangle$.

 Given a splitting $\bbX$ of $G$, and for each $e\in E$, a generating  set $S_e\subset \Gamma_{t(e)}$ for $ Z_{\Gamma_{t(e)}} ( i_e
 (\Gamma_e)) $
 and a family
 $(\gamma_j)_{0\leq j\leq j_0}$ of elements of $G$, one can decide
 whether there is an element $\eta \in Mod_{\bbX}$ whose image
 $\bar{\eta} $ in $\Aut(H_1(G))$  sends $\bar \gamma_j$ in $\bar F$
 for all $j<j_0$ and $\bar \gamma_{j_0}$  inside $\bar t \bar F$.

More precisely, there is such an element $\eta$ if and only if the
explicit Diophantine linear system of equations 
\begin{equation} \label{eq;Dioph_lin}  \forall j, \sum_{e\in E^+\\ s_e \in S_e}  r_{s_e}
  n(\gamma_j,e) \delta(s_e) = -\delta(\gamma_j) \, + \,  {\rm dirac}_{j=j_0}    \end{equation}
 (with unknowns
$r_{s_e}$) has a solution. 
\end{prop}

\begin{proof}

 By $(*)$ (from the proof of Proposition  \ref{prop;Mod_ab}),
$\delta(D_{\epsilon,h}(\gamma)) = \delta(\gamma) + n(\gamma,\epsilon)
\delta(h)$ and by induction,  and
the fact that $\bar h$ is fixed by all Dehn twists, $$\delta(D_{\epsilon_1,h_1} D_{\epsilon_2,h_2}\dots D_{\epsilon_k,h_k}  (\gamma)) = \delta(\gamma) + \sum_{i=1}^k n(\gamma,\epsilon_i) \delta(h_i).$$

Assume that there exists $\eta \in Mod_{\bbX}$ such that $\delta(\eta
(\gamma_j)) =0$  for $j<j_0$ and $=1$ for $j=j_0$. Since oriented Dehn twists
 generate the image of $Mod_\bbX$ in $GL(H_1(G))$, this element
 $\eta$ can be chosen as a product of oriented Dehn twists
 $D_{\epsilon_1,h_1} D_{\epsilon_2,h_2}\dots
 D_{\epsilon_k,h_k}$, which by commutation in $\Aut(H_1(G))$ can be
 chosen to be  $$\prod_{e\in E^+, s_e\in S_e}D^{r_{s_e}}_{e,s_e}   .$$ Therefore by the previous equation, 
 $\forall e \in E, \forall s_e \in S_e, \exists r_{s_e} \in
  \mathbb{Z}$  $$\forall j, \sum_{e\in E^+\\ s_e \in S_e}  r_{s_e}
  n(\gamma_j,e) \delta(s_e) = -\delta(\gamma_j) \, + \,  {\rm dirac}_{(j=j_0)}$$
  (in which   ${\rm dirac}_{j=j_0}$ yields $0$ if $j\neq j_0$ and $1$
  otherwise). 

Conversely, if the system of equations  $$\forall j, \sum_{e\in E^+\\ s_e \in S_e}  r_{s_e}
  n(\gamma_j,e) \delta(s_e) = -\delta(\gamma_j) \, + \,  {\rm dirac}_{j=j_0}$$
  has a solution (in unknowns  $r_{s_e}$), then
  $\delta( \prod_{e\in E^+, s_e\in S_e}D^{r_{s_e}}_{e,s_e}     (\gamma_j)) =0$  for $j<j_0$ and $=1$ for
  $j=j_0$.

We have thus reduced the orbit problem to an equivalent problem of
satisfaction of a system of linear Diophantine equations, that are explicitly computable. The problem is therefore solvable, by classical technics of linear algebra.
\end{proof}

If we had the same statement with $\Aut(G)$ replacing
$Mod_{\bbX}$, we would be very close to our conclusion. 
What we
will  show is that $Mod_{\bbX}$ has finite index image in
$\Out(G)$.

\subsection{No surface vertex group in splittings of suspensions} \label{sec;no_surf}

Our aim is Proposition \ref{prop;no_free}, that states that in any splitting of a suspension, there is no vertex whose group  which is a surface group with boundary, where the boundary subgroups are the adjacent edge groups.

More precisely, if $T$ is a $G$-tree, we say that a vertex stabilizer
$H$ is a hanging surface group if it is the fundamental  group of a
non-elementary compact surface  with boundary components, and the
adjacent edge groups are exactly the subgroups of the boundary
components.     We say that it is a hanging bounded Fuchsian group if
there is a finite normal subgroup $K\normal H $ such that $H/K$  is
isomorphic to the fundamental group of a non-elementary hyperbolic compact $2$-orbifold with boundary, by an isomorphism sending the images of the adjacent edge stabilizers 
on the 
 boundary subgroups   of the orbifold group. 
We will  need the following well known fact about bounded
Fuchsian groups, that we give for
completeness.
\begin{lemma}\label{lem;HBFvirt_free}
   Any  hanging bounded Fuchsian group is virtually free.
\end{lemma}
\begin{proof}
Let $H$ be a  hanging bounded Fuchsian group, and $K$ a finite
subgroup of $H$ as above. 
Because  $H/K$ is the fundamental
 group  of a hyperbolic compact $2$-orbifold, it is  a subgroup of
 $PSL(2,\mathbb{R})\sd \mathbb{Z}/2\mathbb{Z}$, and  by Selberg's lemma, it is residually finite, and it has a finite
 index subgroup $\bar H_0$ without torsion.   This subgroup is the fundamental
 group of a finite cover of the orbifold, which is therefore a surface, with
 boundary, since the orbifold has boundary. Thus, $\bar H_0$ is
 free. It lifts as a free subgroup of $H$, whose
 index in $H$ is the product of indices $      |K|  \times [H/K : \bar H_0]$ .
\end{proof}
 
 We will show (Proposition \ref{prop;no_free}) that a splitting of a
 suspension of a  finitely generated group has no  hanging bounded Fuchsian vertex group.

Let $T$ be the Bass-Serre $G$-tree of a  reduced 
splitting of $G$.   Let us introduce $Y=F \backslash T$ and $X=
G\backslash T$. Both graphs provide respectively graph-of-groups   
decompositions $\bbY$ and $\bbX$ of $F$ and of $G$.

\begin{lemma}
 $Y$ is a finite graph, and the $G$-action on $T$ induces, by
 factorisation,  a $\langle \bar t\rangle $-action on $Y$.
\end{lemma}
\begin{proof} The second part of the statement is obvious since $F$ is
  normal in $G$.
Also, since $F$ is normal in $G$, its minimal subtree in $T$ is $G$-invariant, therefore it is $T$ itself.  
In other words,  $Y$ is its own core. As for its genus, it  it is
finite, since it bounds from below the rank of $F$ (the fundamental group of $\bbY$ at a vertex $v_0$ projects on the fundamental group of the underlying graph $Y$ at $v_0$). It follows that $Y$ is a finite graph, and it is endowed with an action of $\langle \bar t\rangle = G/F$ for which the quotient is $X$. 
\end{proof}

\begin{lemma}\label{lem;vert_are_susp}
  Given any vertex in $T$, its stabilizer in $G$ is a suspension of
  its stabilizer in $F$.      
\end{lemma}

\begin{proof}
Let $\nu$ be such a vertex, and consider $v$ its image in $X$ and  $\tilde v$ its image in  $Y$.
 Because $Y$ is finite, there is a smallest $n>0$ such that $ \bar t^n \tilde v$ is $\tilde v$.  Lifting in $T$, we obtain the existence of   $\gamma \in F$ such that $\gamma t^n \nu = \nu$. Therefore, if  $(F)_{\nu}$ is the stabilizer of $\nu$ in $F$, then it is normalized by $\gamma t^n$ and  $(F)_{\nu}\sd \langle \gamma t^n \rangle$ fixes $\nu$.

We claim that  $(F)_{\nu}\sd \langle \gamma t^n \rangle$  is the stabilizer in $G$ of  $\nu$.

If there is another element in it, it is not in $F$ by definition of $(F)_{\nu}$, and therefore it is some $\eta t^m$, $\eta \in F, m\neq 0$. By minimality of $n$,  $n$ divides $m$, and if $\gamma't^m = (\gamma t^n)^k$, then $\gamma't^m \nu = \eta  t^m \nu $, and $t^{-m}\gamma'^{-1} \eta  t^m $ fixes  $\nu$ and is in $F$ by normality of $F$. Thus it is in  $(F)_{\nu}$, and $\eta t^m \in (F)_{\nu}\sd \langle \gamma t^n \rangle$. This ensures the claim.
\end{proof}

 \begin{lemma}\label{lem;edge_susp}
 For any edge in $T$, its stabilizer in $G$ is a suspension of its
 stabilizer in  $F$.
\end{lemma}
\begin{proof}
Subdivide such an edge (and all edges in its orbit, equivariantly) by inserting a vertex on its midpoint, and apply the previous lemma to this vertex.
\end{proof}

Up to now we have not used any assumption on the nature of edge groups. But now two remarks are of interest, and directly follow from the previous lemmas.

\begin{lemma}\label{lem;YfreesplitsF}
A reduced $G$-tree $T$ never has a finite edge stabilizer.

If $T$ has cyclic edge stabilizers in $G$,  then $\bbY$ is  a free splitting of $F$.  

If $T$ has virtually cyclic edge stabilizers in $G$, then $\bbY$ is a splitting of $F$ over finite subgroups.

\end{lemma}

\begin{lemma}\label{lem;no_free}
Let $\bbX$ be a graph-of-groups decomposition of $F\sd \langle t\rangle$ with virtually cyclic edge groups.   Then no vertex group of $\bbX$  is free non-abelian, or even infinitely many ended.
\end{lemma}

\begin{proof}
By Lemma \ref{lem;YfreesplitsF}, $T$, as an $F$-tree, is the tree of a virtually free splitting of $F$. In particular, since $F$ is finitely generated, each vertex stabilizer in $F$ is finitely generated.  It follows from Lemma \ref{lem;vert_are_susp} that a vertex stabilizer of $T$ in $G$    contains an infinite index normal subgroup, which is as we saw, of finite type. 

Such a group cannot be free.
It cannot be infinitely many ended neither (for the same reason
actually, that we recall in the following lemma).

\begin{lemma}\label{lem;infinitely_many_ends}
 Let $H$ be a finitely generated group with infinitely many ends, then  $H$ has no finitely generated infinite-index normal subgroup.
\end{lemma}

\begin{proof}
 By Stallings' theorem \cite[Thm. 4.A.6.5,  Thm. 5.A.9] {Sta}, $H$ is
 the fundamental group of a reduced finite graph-of-groups with finite
 edge groups (with at least one edge). Let $T$ be the associated
 Bass-Serre tree, and $N$ be a normal subgroup of $H$. Then, the tree
 $T$ is minimal for $N$, hence $T/N$ equals its own core. If $N$ is
 finitely generated, $T/N$ is finite. Moreover, the action of $H$ on
 $T$ factorizes through $T/N$, and if $N$ has infinite index in $H$,
 there is an edge $\tilde{e}$ in $T/N$ fixed by  infinitely many
 different elements of $H$, $\{h_i, i\in I\}$, all in different
 $N$-coset. Let $e$ its image in $T/H$ and $\tilde{\tilde{e}}$ a
 choice of lift in $T$. There are $n_i \in N$ such that for all $i$,
 $h_in_i$ fix $e$. By finiteness of edge stabilizers, infinitely many
 of the elements $h_in_i$ are equal, 
contradicting that the $h_i$ were in different $N$-cosets. 
\end{proof}

\end{proof}

\begin{prop}\label{prop;no_free} Let $F$ be a finitely generated group, and $G=F\sd \langle t\rangle$ a suspension.
 Given any graph-of-groups decomposition  $\bbX$ of $G$,  no vertex
 group of $\bbX$ is a hanging surface group, 
 or a  hanging bounded Fuchsian group. 
\end{prop}

\begin{proof}
Assume the contrary: let $\Gamma_{v_0}$ be an alleged  hanging
bounded Fuchsian group or  hanging surface group, which is, in
particular, virtually free (Lemma \ref{lem;HBFvirt_free}), hence
infinitely-many ended because it is non-elementary.
 Because it is hanging, all the neighboring edges of $v_0$ carry
 virtually cyclic groups. Collapse all other edges in $\bbX$ in order
 to get $\bbX'$, whose edges are virtually cyclic. The image $v_0'$ of
 $v_0$ carries the same group, since no adjacent edge has been
 collapsed, and this group is infinitely-many ended, as we noticed.
 Apply Lemma \ref{lem;no_free} to $\bbX'$ to get the contradiction.
\end{proof}

\subsection{The $\Aut(G)$-orbit of the fiber in the hyperbolic case.}

Let us recall that the canonical \ZmJSJ  decomposition 
 of a one-ended hyperbolic
group is a certain finite splitting $\mathbb{X}$ of $G$ over certain virtually cyclic
subgroups (maximal  with infinite center), such that every automorphism of $G$ induces an automorphism
of graph of groups of $\mathbb{X}$ (see \cite[section
4.4]{DGu_gafa}). 
In other word, the natural map
$\delta\Aut(\bbX) \to \Out(G)$ is surjective.

\begin{remark} The choice to use the rather technical
  \ZmJSJ  splitting instead of the more natural ``virtually-cyclic''
  JSJ-splitting, is only suggested by our ability to algorithmically
  compute this decomposition. In principle, we could work with the
  classical JSJ splitting as well in the same way.
\end{remark}

Let $\bbX$ be a graph-of-groups, $v$ a vertex therein,  and $\Gamma_v$ the vertex group.  The choice of an order on the   oriented edges adjacent to $v$, and of a generating set of the edge groups, endows $\Gamma_v$ with a {\it marked peripheral structure}, that is the tuple of conjugacy classes of the images of these generating sets by the attaching maps. We denote by $\calT$ this tuple, and $\Out_m(\Gamma_v,\calT)$ the subgroup of $\Out(\Gamma_v)$ preserving $\calT$ (see also \cite{DGu_gafa}). In the following, the choices of order and generating sets are implicit, and done {\it a priori}.

The following is a typical feature of a JSJ decomposition of a hyperbolic group, whose proof, in this specific setting, is essentially contained in \cite{DGu_gafa}.

\begin{lemma}\label{lem;ZmJSJ} Let $G$ be a one-ended hyperbolic group.
 Let $(\Gamma_v, \calT)$ be a vertex group of the
 \ZmJSJ  decomposition $\bbX$ of $G$, with the marked peripheral structure
 induced by $\bbX$ (and some choice of finite generating sets of edge
 groups). 
 If $\Out_m(\Gamma_v, \calT)$ is
 infinite, then $G$ admits a splitting
 with a hanging bounded Fuchsian vertex group. 
\end{lemma}

\begin{proof}
By \cite[Prop. 3.1]{DGu_gafa}, such a vertex group $\Gamma_v$ must
have a further compatible splitting over a maximal virtually cyclic
group with infinite center, which allows to use  \cite[Prop
4.17]{DGu_gafa} to ensure that $(\Gamma_v, \calT)$ is a so-called
hanging orbisocket, which by definition \cite[Def. 4.15]{DGu_gafa}
allows to refine $\bbX$ in order to get a splitting of $G$ whose one
vertex group is a hanging bounded Fuchsian group.\end{proof}

\begin{prop}\label{prop;coset_rep}
  Let $F$ a finitely presented group, and $G=F\sd \langle t\rangle$ a
  suspension that is assumed to be hyperbolic.  

  Then, the image in $\Out(G)$ of the small modular group of the \ZmJSJ  decomposition of $G$
  has finite index in $\Out(G)$. Moreover, one can compute a set of
  right-coset representatives of $Mod_{\bbX}$ in $\Out(G)$ (in the
  form of automorphisms of $G$).

\end{prop}

\begin{proof}
Since $\delta\Aut(\bbX)$  surjects on $Out(G)$ in the case of the
\ZmJSJ    decomposition, 
it suffices to show that the small modular group has finite index in
$\delta\Aut{\bbX}$ and that
coset representatives can be computed in $\delta\Aut{\bbX}$. 

Once again, this is essentially done in \cite{DGu_gafa} (and probably in other
places).

First, the splitting $\bbX$ can be effectively computed \cite[Prop. 6.3]{DGu_gafa}.

  Consider the following three maps. First $q_X: \delta \Aut(\bbX) \to
\Aut(X)$ where $\Aut(X)$ is the automorphism group of the underlying
finite graph $X$. Second, the natural map $q_E:\ker q_X\to \prod_{e\in E} \Aut (\Gamma_e)$.
 Third, the natural map $q_V:\ker q_E\to \prod_{v\in V} \Out_m
 (\Gamma_v, \calT)$, where $\calT$ is the marked peripheral
 structure induced by the ambient graph-of-groups $\bbX$.

The group $\ker q_V$ is the small modular group. The two first maps
have finite image. By Lemma \ref{lem;ZmJSJ} and Proposition \ref{prop;no_free},
 the map $q_V$ also has finite image.

 Therefore the small modular group has finite
index in $\delta \Aut(\bbX)$.  In order to
reconstruct coset representatives of the small modular group in
$\delta \Aut(\bbX)$, it is enough to find coset
representatives for the kernel of each of these maps.

In order to compute coset representatives of $\ker q_E$ in $\delta
\Aut(\bbX)$, we can make the finite list of all graph automorphisms
 of $X$ for which $\Gamma_{\Phi_X(e)} \simeq \Gamma_e$. Let $\Phi_X$
 be any of them.
We can make the list of all isomorphisms $\Gamma_{\Phi_X(e)}
\simeq \Gamma_e$ (there are finitely many such isomorphisms since
these groups are virtually cyclic).
Then we can  apply \cite[Prop. 2.28]{DGu_gafa} in
order to reveal whether this automorphism $\Phi_X$ has a preimage by $q_X$.

In order to compute coset representatives of $\ker q_E$ in $\ker q_X$,
we consider a collection of automorphisms of edge groups, and  apply
\cite[Prop. 2.28]{DGu_gafa} in order to reveal whether this collection has a preimage by $q_E$.

Finally, in order to compute coset representatives of $\ker q_V$ in
$\ker q_E$,  one can make the list of all elements in $\Out
(\Gamma_v,\calT)$ expressed as automorphisms  (by
\cite[Coro. 3.5]{DGu_gafa}, we can enumerate all of
them) and for
each choice of them (for each $v$), check  whether the
collection defines a graph-of-groups automorphism by solving the
simultaneous conjugacy problem that allows the diagram \ref{diagram;Bass} to commute.

\end{proof}

\begin{coro}\label{coro;orbit_fiber_Aut}

Let $F$ be finitely presented, and $G=F\sd \langle t\rangle$  a
suspension that is assumed to be hyperbolic. 
Given 
a    
splitting $\bbX$ of $G$, generating sets for the centralizers of
adjacent edge groups in vertex groups, and a family $(\gamma_j)_{0\leq
  j\leq j_0}$ of elements of $G$, one can decide whether there is an
element $\eta \in \Aut(G)$ whose image $\bar{\eta}$ in $\Aut(H_1(G))$
sends $\bar \gamma_j$ in $\bar F$ for all $j<j_0$ and $\bar
\gamma_{j_0}$  inside $\bar t \bar F$.
\end{coro}

\begin{proof}
First we take note that  $\eta \in \Aut(G)$ satisfies the conclusion
if and only if any other automorphism in the same class in $\Out(F)$
satisfies it. Thus,  let $\alpha_1,\dots,\alpha_k$ be the right-coset representatives of
$Mod_\bbX$ in $\Out(G)$ computed by Proposition
\ref{prop;coset_rep}, in the form of automorphisms. 
 For each $i$, we compute, for each $j$,
$\alpha_i(\gamma_j)$, and we use Proposition \ref{prop;orbit_fiber} in order to
decide whether there is $\eta_0\in Mod_\bbX$ such that
$\bar{\eta_0} \overline{ (\alpha_i(\gamma_j) )} \in \bar F$ for all $j<j_0$ and
$\bar{\eta_0} \overline{(\alpha_i(\gamma_{j_0}) )} \in \bar t\bar F$.

If there exists an index $i$ such that the  answer is positive, then
$\overline{\eta_0 \circ \alpha_i}$ sends $\bar \gamma_j$ in $\bar F$ for all $j<j_0$ and $\bar
\gamma_{j_0}$  inside $\bar t \bar F$.

If for all $i$ the answer is negative, then no automorphism of $G$
satisfies this property.

\end{proof}

\section{Conjugacy and suspensions}

The following observations elaborate on some well known point of view
(see for instance \cite{Sel}, \cite{Arzh}...), and, as stated in the
introduction,  is our angle of attack of the conjugacy problem.

\begin{lemma}\label{lem;wellknown} Let $\phi_1$ and $\phi_2$ be two
  automorphisms of $F$. The following assertions are equivalent.
\begin{enumerate}
 \item \label{it_1} $\phi_1$ and $\phi_2$  are conjugate in $\Out(F)$;
 \item  \label{it_2} there is an isomorphism between their suspensions that
   preserves the fiber  (in both directions) and the orientation;
 \item  \label{it_3} there is an isomorphism between their suspensions that
   preserves the orientation and sends the fiber inside the fiber;
 \item \label{it_4} there is an isomorphism between their suspensions whose
   factorization through the abelianisations preserves the orientation,
   and sends the image of the
   fiber inside the image of the fiber. 
\end{enumerate}
\end{lemma}

\begin{proof}
Of course, \ref{it_2} implies \ref{it_3}, which implies
\ref{it_4}.
Assuming \ref{it_4} we now show \ref{it_3}. If the given isomorphism
$\psi : F\sd_{\phi_1} \langle t \rangle \to F\sd_{\phi_2} \langle t \rangle$
does not preserve the fiber, it sends some $f\in F$ on some $f't^k, f'\in
F, k\neq 0$. 
 Since the fibers are kernels of some cyclic quotient, the derived subgroups of the suspensions are
contained in the fibers. Thus  the image of  $f't^k$ in the abelianisation
of $F\sd_{\phi_2} \langle t \rangle$ is not in the image of $F$. thus, 
 the factorisation through abelianisations of $\psi$ does not send the
 image of the fiber inside the image of the fiber. 
If $\psi$ does not preserve the orientation, it sends $t$ to some $f't^k, f'\in
F, k\neq 1$, and the same argument shows that    the factorisation
through abelianisations of $\psi$ does not  preserve the orientation. 
 Thus we obtain that \ref{it_4} implies
\ref{it_3}.  
Let us prove that  \ref{it_3} implies  \ref{it_2}. 
Let $\alpha$ the isomorphism given by  \ref{it_3}. Then $\alpha(F)$ is
normal in 
$(F\sd_{\phi_2} \langle t \rangle)$ and the quotient is infinite cyclic. Thus
the image of the fiber $F$ is trivial in this quotient (because the
further quotient by this image is also infinite cyclic). It follows
that $\alpha(F) = F$.

Let us prove that \ref{it_1} is equivalent to \ref{it_2}.
Assume that $\Psi : F\sd_{\phi_1} \langle t\rangle \to F\sd_{\phi_2} \langle t'\rangle$  sends $F$ to $F$, and $t$ to $f_0t'$. Then write $\psi$ for the restriction of $\Psi$ to $F$.   

In  $F\sd_{\phi_1} \langle t\rangle$, for all $f\in F$,  one has $t^{-1} f t = \phi_1(f)$. Passing through $\Psi$, one gets (in  $F\sd_{\phi_2} \langle t'\rangle $)  $t'^{-1} f_0^{-1}  \psi(f) f_0t' = \psi( \phi_1(f))$, that is $\phi_2\circ \ad_{f_0} \circ \psi = \psi \circ \phi_1$.

Thus, the classes of $\phi_1$ and $\phi_2$ are conjugate in $\Out(F)$ and furthermore, if $f_0= 1$, $\phi_1$ and $\phi_2$ are conjugate in $\Aut(F)$. 

Conversely, if $\phi_1 = \psi^{-1} \circ   \ad_{f_0} \circ \phi_2 \circ \psi$ for some $\psi$ in $\Aut (F)$, one can extend $\psi$ to $\tilde{\Psi}: F*\langle t\rangle \to  F\sd_{\phi_2} \langle t\rangle$   by setting $\Psi(t) = f_0t \in   F\sd_{\phi_2} \langle t\rangle$. The relation of the semi-direct product by $\phi_1$ vanishes in the image, thus $\tilde{\Psi}$ factorizes through $F\sd_{\phi_1} \langle t\rangle$ producing a bijective morphism.  

\end{proof}

\begin{thm}\label{thm;decide_fop_iso_for_hyp}

There is an algorithm that, given $F$ a finitely  presented group, 
and two automorphisms $\phi,\phi'$ of $F$ such that the suspensions
are  word-hyperbolic, decides whether  $\phi$ and $\phi'$ are
conjugated in $\Out(F)$.

\end{thm}

\begin{proof}
By Lemma \ref{lem;wellknown}, it suffices to decide whether the
associated semi-direct products of $F$ with $\mathbb{Z}$ with
structural automorphisms $\phi$ and $\phi'$ are isomorphic by an
isomorphism satisfying characterization \ref{it_4} in Lemma \ref{lem;wellknown}.

Let $F'$ be another copy of $F$, with same presentation. We read
$\phi'$ as an automorphism of $F'$.
Let us denote by $G$ and $G'$ the groups of the suspensions of $F
$ and $F'$ by
the given $\phi$ and $\phi'$ respectively. Provided with a presentation of
$F$, we have presentations of $G$ and $G'$.

By the main result of \cite{DGu_gafa}, we can decide whether there is
an isomorphism between $G$ and $G'$. If there is none, we are done. If
there is one, say $\psi: G\to G'$, any other isomorphism is in the
orbit of $\psi$ by $\Aut(G')$. Let  $\{f_i, i=1,\dots,j_0-1\}$ be a generating set
of $F$. We apply our solution to the orbit  
problem \ref{coro;orbit_fiber_Aut} to the elements $\gamma_i = \psi(f_i)$ for $i <j_0$, and
$\gamma_{j_0} = \psi(t)$. By definition, the answer to this orbit problem is positive
if and only if there is an automorphism $\eta$ such that $\eta \circ
\psi$ satisfies  characterization \ref{it_4} in Lemma
\ref{lem;wellknown}. Since all isomorphisms $G\to G'$ are of this
form ({\it i.e.  } $\eta \circ
\psi$ for some automorphism $\eta$), this decides whether there is an isomorphism satisfying
the assertion (\ref{it_4}) of Lemma \ref{lem;wellknown}, hence, whether  $\phi$ and $\phi'$ are
conjugated in $\Out(F)$.

\end{proof}

Theorem \ref{thm;decide_fop_iso_for_hyp} covers the case of atoroidal
automorphisms of free groups, by \cite{Br_gafa}.

\begin{coro} Let $F$ be a free group. 
 The conjugacy problem in  $\Out(F)$ restricted to atoroidal  automorphisms is solvable. 
\end{coro}

{\footnotesize

\vskip 1cm

\noindent {\sc Fran\c{c}ois Dahmani, \\ Univ. Grenoble Alpes, Institut Fourier UMR5582, F-38402 Grenoble, France.
\\{\tt e-mail: francois.dahmani@ujf-grenoble.fr} 

}}

 %\end{document}

\newpage

\part{On suspensions, and conjugacy of a few more  automorphisms of
  free groups.} \label{part2}

\setcounter{section}{0}
\begin{abstract} 

In a previous work, we remarked that  the conjugacy problem for pairs
of atoroidal automorphisms of a free group was solvable by mean of
the isomorphism problem for hyperbolic groups and an orbit problem for
the automorphism group of their suspensions ({\it i.e.} their
semidirect product with $\bbZ$ for the relevant structural automorphism). 

We consider the same problem a few more automorphisms of free groups, those that
produce relatively hyperbolic suspensions  that do not split over a
parabolic subgroup.
\end{abstract}

\subsection*{Introduction}

Let $F$ be a finitely presented group (we will soon assume that it is
free), $\Aut(F)$ be its automorphism group,  and $\Out(F) = \Aut(F)/\Inn (F)$ be its outer automorphism group.

 Given two semi-direct products, $F\sd_\alpha \langle t \rangle$  and
 $F\sd_\beta \langle t' \rangle$, their structural automorphisms
 $\alpha$ and $\beta$ are conjugated in $\Out(F)$ if and only if there
 is an isomorphism $F\sd_\beta \langle t \rangle   \to  F\sd_\alpha
 \langle t' \rangle$  that  preserves the fiber (which is $F$) and the
 orientation ({\it i.e.} sends $tF$ on $t'F$). 

This suggests a way of analysing the conjugacy problem in  a class of
elements of $\Out(F)$
through an isomorphism problem in a class of semidirect products of
$F$.

A motivating case is that of a free group. Though a solution to the conjugacy problem of automorphisms of free
groups was  announced
by Lustig \cite{Lu1, Lu2}, it might still be desirable to find
short complete solutions for specific classes of elements in
$\Out(F_n)$.

In \cite{DCPout} we considered the case of atoroidal automorphisms. In
that case, the semi-direct product $F\sd_\alpha \langle t \rangle$
is a hyperbolic group \cite{Br_gafa}, and since the isomorphism problem for
hyperbolic groups is solvable \cite{Sel, DGr, DGu_gafa},   the conjugacy
problem  for atoroidal automorphisms of free groups was reduced to an
orbit problem for $\Out( F\sd_\alpha \langle t \rangle)$ in
$H_1(F\sd_\alpha \langle t \rangle, \bbZ)$
(this orbit
problem was an interpretation of the condition that there
 is an isomorphism $F\sd_\beta \langle t \rangle   \to  F\sd_\alpha
 \langle t' \rangle$  that  preserves the fiber and the orientation,
 once is given an abstract isomorphism $F\sd_\beta \langle t \rangle   \to  F\sd_\alpha
 \langle t' \rangle$).  We solved this orbit problem by showing that
 $\Out( F\sd_\alpha \langle t \rangle)$ is a virtually abelian group,
 and by interpreting the orbit problem as a system of linear
 Diophantine equations.

In view of \cite{DGr}, \cite{DG_CbDF}, it is natural to ask whether one can approach the 
conjugacy problem of larger classes of automorphisms, namely those producing
proper relatively hyperbolic suspensions.

\begin{defi}\label{def;autom;relhyp} Let $\phi\in \Aut (F)$, and $F_1,
  \dots, F_k$ finitely generated proper subgroups of $F$. We say that the automorphism $\phi$ is hyperbolic relative to $\{F_1, \dots, F_k\}$ if there exists integers $m_1, \dots, m_k>0$ and elements $f_1, \dots, f_k \in F$ such 
that, for all $i$,   $t^{m_i} f_i$ normalises $F_i$,       and such that the group $(F\sd_\phi \langle t \rangle)$ is hyperbolic relative to $\{(F_i \sd\langle t^{m_i} f_i \rangle), i=1, \dots, k  \}$. 
\end{defi}

The case of automorphism of free groups  is once again particularly
interesting, since according to \cite{GL}, all non-polynomial automorphisms of free 
groups should produce interesting relatively hyperbolic suspensions
(of course it could be interesting to consider also a free product of nice
groups).

 Thus from now on $F$ is a free group.

One says that a subgroup $F_0$ of $F$ is \emph{polynomial} for a given
automorphism $\phi$ if every conjugacy class of elements in $F_0$ has
polynomial growth under iterates of $\phi$ (more explicitly, that
means that for all $\gamma\in F_0$, the length of a cyclically reduced
representative of $\phi^n(\gamma)$ is bounded above by a polynomial in
$n$). 
We say that an automorphism is polynomial if $F$ itself is
polynomial.     For an {\it outer} automorphism $\Phi$ of $F$, we say that a
subgroup $F_0$ of $F$ is polynomial for $\Phi$ if there is an automorphism
$\phi$ in the class of $\Phi$ for which $F_0$ is polynomial.

In \cite[Prop 1.4]{Lev09} Levitt proves that for any outer
automorphism $\Phi$ of a free group $F$, there is a finite family of
finitely generated  subgroups of $F$, polynomial for $\Phi$,
such that all polynomial
subgroups of $\Phi$ are conjugated into one of them (see also  \cite[Prop. 3.2]{GL}).

The aim of this note is  thus to explore to what extend the method used in
\cite{DCPout} can be extended to larger classes of automorphisms of
free groups, and in particular to (some) non-polynomial automorphisms.

 However, I ultimately had to restrict the study to those
 automorphisms whose suspension does not split over a parabolic
 subgroup. I also have to concede that this attempt  
uses three results unpublished at the time of writing (this issue will be
made clear in a few lines).

The main result of this attempt is the following.
\begin{theo}
There is an (explicit) algorithm that, given two automorphisms
$\phi_1, \phi_2$ of a finitely generated free group $F$, terminates 
if both produce proper
relatively hyperbolic suspensions, relative to suspensions of
polynomial subgroups, without  parabolic splitting,  and it indicates
whether $\phi_1$ and $ \phi_2$ are conjugated in $\Out(F)$.
\end{theo} 

The arguments presented below involve other tools than in
\cite{DCPout},  in
particular Dehn fillings, and growth of conjugacy classes under
iterations of automorphisms.   
 They rely on a certain number
of currently unpublished results, so I would like to make this
reliance clear. First, there is the main result of Gautero and Lustig
paper \cite{GL}. This is used twice;  
  to produce examples to which the results might apply (so, in some
  sense, as a motivation), and to compute explicitly the polynomial
  subgroups (actually, this is to certify that an exponentially
  growing automorphism is indeed exponentially growing). 
 Then there is the splitting computation of Touikan \cite{Tou}. And finally,
there is the solution ot the isomorphism problem of some rigid relatively
hyperbolic groups, by Guirardel and myself, \cite{DG_CbDF}. 

The authors wants to say merci to the referee, for useful comments and
corrections.

\section{Preliminary}

\subsection{General}
\label{para;recall}

Since this note is a sequel to \cite{DCPout}, we assume that the
reader has access to that previous paper, and we will freely use
its content. For readability, though, we briefly introduce now a few
items that we need
from that paper. First is a variation on some classical fact.

As in \cite{DCPout} we will call a semidirect product with $\bbZ$,  $F\sd_\phi \bbZ$
a suspension of $F$ by $\phi$, whose fiber is $F$ and whose
orientation is defined by $Ft$ (this is to distinguish it from the suspension
by $\phi^{-1}$ which is the same group, but with reverse orientation). 
  
\begin{lemma} (see for instance \cite[Lemma 2.3]{DCPout}) \label{lem;wellknown} Let $\phi_1$ and $\phi_2$ be two
  automorphisms of $F$. The following assertions are equivalent.
\begin{enumerate}
 \item \label{it_1} $\phi_1$ and $\phi_2$  are conjugate in $\Out(F)$;
 \item  \label{it_2} there is an isomorphism between their suspensions
   $(F\sd_{\phi_1} \langle t \rangle)$ and $(F \sd_{\phi_2} \langle t' \rangle)$ that
   preserves the fiber  $F$ (in both directions) and the orientation
   ({\it i.e} sends $t$ in $Ft'$);
 \item  \label{it_3} there is an isomorphism between their suspensions that
   preserves the orientation and sends the fiber inside the fiber;
 \item \label{it_4} there is an isomorphism between their suspensions whose
   factorization through the abelianisations preserves the orientation,
   and sends the image of the
   fiber inside the image of the fiber. 
\end{enumerate}

\end{lemma}
 
When talking about a splitting of a group, we mean a graph-of-group
decomposition (often noted $\bbX$) as defined in Bass-Serre
theory. This is presented in many places, beginning with Serre's
famous book. We recall very briefly our conventions for defining a
splitting, and its automorphisms. 
One is given an underlying graph (unoriented, with possible double edges,
and loops)  $X$ whose set of   vertices and set and oriented edges we
denote respectively by $V$ and $E$, and whose involution on the
oriented edges (reversion of orientation) we denote by $\ol{ } :
e\mapsto \ol{e}$, and terminaison map $t:E\to V$.   One is given
groups for each vertices, denoted $\Gamma_v, v\in V$, and for each
edge $e\in E$,  another group    $\Gamma_e$,  with $\Gamma_e =
\Gamma_{\bar e}$, and an injective morphism $i_e:
\Gamma_e \hookrightarrow \Gamma_{t(e)}$.
 The Bass group is the group generated by all vertex groups and all
 edges with the relations that $\bar e = e^{-1}$ and that $\bar e i_{\bar
   e} (g) e  = i_e (g)$ everywhere it is defined.  The
 fundamental group of the graph of group at a vertex $v_0$ is the
  subgroup of the Bass group 
   of all elements of the form $e_1 \gamma_1
  e_2 \gamma_2 \dots e_r \gamma_r$ where $e_i \in E$,
  $\gamma_i\in \Gamma_{t(e_i)}$ for all $i$, and such that
  consecutive $e_i$ define a loop at $v_0$ in the graph $X$ (that is,
  for all $i$, $t(\ol{e_{i+1}}) = t(e_i)$, and $t(e_r) = t(\ol{e_1}) = v_0$).

If this fundamental group (of the graph of group $\bbX$) is
 isomorphic to a certain group $G$ we say that $\bbX$ is a splitting
 of $G$.  We call a splitting non-trivial if the action of its
fundamental group on the
Bass-Serre tree has no global fixed point. 

An automorphism of the graph of groups $\bbX$ is a tuple $(\Phi_X,
(\phi_v), (\phi_e), (\gamma_e))$ where $\Phi_X$ is a automorphism of
the underlying graph $X$, for all vertices $v$, $\phi_v:\Gamma_v \to
\Gamma_{\Phi_X(v)}$ is an isomorphism,  for all edges $e$, $\phi_e:\Gamma_e \to
\Gamma_{\Phi_X(e)}$ is also an isomorphism, and
$\gamma_e \in \Gamma_{\Phi_X(t(e))}$
satisfies  \begin{equation} \label{BassDiagram}  \begin{aligned}
    \text{Bass Diagram:} &  & \phi_{t(e)}\circ
i_e = {\rm ad}_{\gamma_e} \circ i_{\Phi_X(e)} \circ
\phi_e, \end{aligned}\phantom{\hspace{3cm}}   \end{equation}
for ${\rm ad}_{\gamma_e} : x\mapsto \gamma_e^{-1} x \gamma_e$.  One
might like to read the condition as: ``each attaching map $i_e, e\in E$ commute
with the isomorphisms $\phi_v, \phi_e, v\in V, e\in E$ up to
conjugation in the target vertex group''.

 The small modular group of a splitting  was used in \cite{DCPout}. We
 suggest reading \cite[\S1]{DCPout} and more precisely 
 \S 1.2 {\it loc. cit.} for a slightly broader discussion about it.  The small
 modular group  of a splitting $\bbX$
of a group $G$, denoted by  $Mod_{\bbX}$    is a subgroup of the
automorphism group of $\bbX$ consisting of those for which $\Phi_X = {\rm Id}_X$,
 $\phi_v \in \Inn (\Gamma_v)$ and  $\phi_e ={\rm Id}_{\Gamma_e}$ for
 all $v,e$.  It is generated by  the union of two families of
 automorphisms,  the oriented Dehn twists (for which the $\phi_v$ are
 all the identity, and only one element $\gamma_e\in Z_{\Gamma_{t(e)}}
 (i_e( \Gamma_e) )$ is non trivial), and the
inert twists (for which $\phi_v = {\rm ad}_{\gamma_v}$ and $\gamma_e $
is the same
$\gamma_v$ if $t(e)=v$). The  image of the small modular group in the outer automorphism group of $G$
consists of the group generated by Dehn twists over 
edges of the splitting $\bbX$.   Its  image in the automorphism group of the abelianisation of  $G$ is generated by Dehn twists over non separating
edges of  $\bbX$. 

Given a suspension $F\sd \langle t \rangle$, we define the map
$\delta: F\sd \langle t \rangle \to \bbZ$ to be the quotient by
$F$. Of course, it factorises through the abelianisation of $F\sd \langle t
\rangle$. We insist in seeing the targer of $\delta$ as $\bbZ$ to be
able to interpret $\delta(\gamma)$ as an  integer.  

Given a splitting $\bbX$ of $F\sd \langle t \rangle$, and a choice of
base point in $\bbX$, we can realise each element $\gamma$ of  $F\sd \langle t
\rangle$ as its expression in the Bass group, and for each (oriented)
edge $e$ of $X$ we may define $n(\gamma,e)$ as the number of occurences
of $e$ in the reduced form of this expression, minus the number of
occurences of $\ol{e}$.

We obtained the following result.

\begin{prop} (See \cite[2.3]{DCPout})\label{prop;orbit_fiber}
  Let $G$ be a finitely generated group that can be expressed as a
  semi-direct product $F \sd \langle t\rangle$.

 Given a splitting $\bbX$ of $G$, and for each $e\in E$, a generating  set $S_e\subset \Gamma_{t(e)}$ for $ Z_{\Gamma_{t(e)}} ( i_e
 (\Gamma_e)) $
 and a family
 $(\gamma_j)_{0\leq j\leq j_0}$ of elements of $G$, one can decide
 whether there is an element $\eta \in Mod_{\bbX}$ whose image
 $\bar{\eta} $ in $\Aut(H_1(G))$  sends $\bar \gamma_j$ in $\bar F$
 for all $j<j_0$ and $\bar \gamma_{j_0}$  inside $\bar t \bar F$.

More precisely, there is such an element $\eta$ if and only if the
explicit Diophantine linear system of equations 
\begin{equation} \label{eq;Dioph_lin}  \forall j, \sum_{e\in E^+\\ s_e \in S_e}  r_{s_e}
  n(\gamma_j, e) \delta(s_e) = -\delta(\gamma_j) \, + \,  {\rm dirac}_{j=j_0}    \end{equation}
 (with unknowns
$r_{s_e}$) has a solution. 
\end{prop}

\subsection{On polynomial growth}

The following  preliminary result is useful. It follows from the
recent algorithmic construction of relative train tracks, by Feighn
and Handel \cite[Theorem 2.1]{FH14}.  We give a different proof below 
(certainly not of the same scope as the mentioned reference) for the curiosity of the reader.

\begin{prop}\label{prop;certify_poly}

There is an algorithm that, provided with a free group $F$ and an automorphism $\phi$, terminates and indicates whether $F$ is polynomial for $\phi$. 

\end{prop}

\begin{proof}

First, we will give a procedure certifying that an automorphism is
polynomial, and then a procedure certifying that an automorphism is of
exponential growth on some conjugacy class. 

By \cite[Coro 5.7.6]{BFH1_Tits1},  if $\phi$ is a polynomially growing automorphism, then there is $n$ such that $\phi^n$ is unipotent in $GL(H_1(F))$.
It is then sufficient to devise a procedure certifying whether a unipotent automorphism is polynomially growing. Then we use (a weak aspect of) Theorem \cite[3.11]{BFH2_Tits2}: if the automorphism is polynomially growing, there exists a topological representative $\tau: G\to G$ of $\phi$ on a graph $G$, and a filtration of $G$, $\{v\}=G_0\subset \dots \subset G_n=G$  such that any edge $e$ in $G_i\setminus G_{i-1}$ is sent on a path $ec$ where $c$ is a path in $G_{i-1}$.  
 Also, 
if such a representative exists, then, for every edge $e \in G_i$,  the length of $\phi^n(e)$  can be bounded by a polynomial  
in $n$ depending only on $i$   (this can be seen by induction; it is obvious for $i=0$ or $1$ since $G_0$ contains no edge, and if it is true for $i-1$, let $P_{(i-1)}$ the corresponding polynomial,  and for  $e \in G_i$,  with $\phi(e)= ec$, we can write $\phi^n(e)= e\phi(c)\phi(c)^2\dots \phi(c)^{n-1}$,  and the total length is bounded by $\sum_{k\leq n-1}  P_{(i-1)}(k)^{|c|}$, hence by  $\sum_{k\leq n-1}  P_{(i-1)}(k)^{M}$ for $M=\max\{|\phi(e)|, e\in G_i\}$, which is polynomial in $n$).
Thus, $\phi$ is polynomial if and only if there is such a topological
representative. This can be certified by enumeration of topological
representatives, since the condition used is easily algorithmically
checked.

We now need a procedure that produces a certificate that $\phi$ is not
polynomialy growing when it is the case.

 For that, we'll use that $\phi$ is not polynomially growing if and only if
 the suspension has a proper relative hyperbolic structure. One
 direction of this equivalence ($\implies$) is the content of  \cite{GL} (actually  \cite{GL} describes the relative hyperbolic structure). We
 present now an argument for the other direction. 
  If the suspension is a proper relatively hyperbolic
group (with at least one hyperbolic element), then there are
hyperbolic elements in each coset of the fiber: this follows from
\cite[Lemma 4.4]{OsinIJAC} (I also find
 rather pleasant the following proof:   a simple random walk on the relatively hyperbolic group
 $F\sd \bbZ$ will a.s. walk on only finitely many non-hyperbolic
 elements (apply Borel-Cantelli Lemma, with the exponentially
 decreasing probability to walk on a parabolic element,
 e.g. \cite{Sis}), and, by recurrence on $\bbZ$, it
 will walk infinitely many times on the preimage of any chosen coset).   
Thus, there is a hyperbolic element $f_h$ in the
 fiber $F$, and another $t'= tf_0$ in the coset $tF$. 
  
  The relative distance of $F$ and $t'^kF$ grows therefore linarly in $k$, and by
 exponential divergence (in the hyperbolic coned-off graph), 
 the shortest path, in  $t'^kF$, from  $t'^{k}$ to  $f_h t'^k$ has
 exponential relative length in $k$, hence exponential absolute
 length. This makes $f_h$ an exponentially growing element
 for the automorphism $\phi\circ \ad_{f_0}$. This means that this
 automorphism cannot be polynomialy growing (as it is visible on the
 topological representative $\tau$ of a polynomially growing
 automorphism that no element can be exponentially growing).
 Therefore $\phi$ has an exponentially growing \emph{conjugacy class}
 in $F$.

By \cite{DG_PPS}, and enumeration of the proper subgroups of $F$,
if there exists a proper relative hyperbolic structure, one can
eventually find it and thus a certificate that the automorphism is not
polynomially growing (here a certificate is the data of an
exponentially growing conjugacy class, with a proof that it is
exponentially growing).

\end{proof}

\begin{prop}\label{prop;compute_RHS}
 Let $F$ be a free group. There is an explicit algorithm that,
 given $\phi \in \Aut(F)$ expressed on a basis of $F$, terminates and
 produces a basis for each group of a collection $F_1, \dots, F_k$ of
 maximal 
 polynomial subgroups of
 $F$ for $\Phi$ the class of $\phi$ in $\Out(F)$, and computes minimal exponents $m_i>0$ and elements $f_i$ so
 that $t^{m_i} f_i$ normalises $F_i$ (see definition \ref{def;autom;relhyp}).  
\end{prop}

\begin{proof}
Note that there is a unique relative
hyperbolic structure for $F\sd_{\phi} \langle t \rangle$ whose
parabolic groups are suspensions of polynomial subgroups. Indeed, the
polynomial subgroups must be parabolic, by the observation made in the proof of
\ref{prop;certify_poly}.
Moreover, by \cite{GL}, $F\sd_{\phi} \langle t \rangle$ is indeed
relatively hyperbolic to the subgroups that we need to compute.

We enumerate the tuples $(S,m,f)$, where $S$ is a finite subset of
$F$, $m$ is an integer, and $f\in F$.  For each of them, by the usual Stallings' folding process, we may find a
basis of  $\langle S\rangle$,  and we may check
whether $\langle S\rangle$ is stable by conjugation by  $t^mf $ and
$f^{-1}t^{-m}$. If so,  $\langle S\rangle$  is normalized by  $t^mf $;
in that case,   
 using Proposition \ref{prop;certify_poly}, we may certify whether,
 the product $\langle S\rangle \sd \langle t^mf \rangle$  is a suspension of a
polynomial automorphism on $\langle S\rangle $. For any collection of such subgroups, we may use \cite{DG_PPS} in order to certify that $F\sd_{\phi} \langle t \rangle$ is relatively hyperbolic. When this happens, the algorithm is done.
\end{proof}

Recall that a splitting of a relatively hyperbolic group is peripheral if, in the Bass-Serre tree, all parabolic subgroups are elliptic. 
Let us say that $\phi\in \Aut(F)$ is relatively hyperbolic with no
parabolic splitting (RH-noPS for short) if it is  properly hyperbolic
relative to a collection of polynomial subgroups of $F$, and the
suspension $F\sd_\phi \langle t\rangle$ has no non-trivial  
peripheral  splitting over
a subgroup of a parabolic subgroup.  

Let us say that $\phi\in \Aut(F)$ is relatively hyperbolic with no elementary splitting (RH-noES for short) if it is properly hyperbolic relative to a collection of polynomial subgroups of $F$, and the suspension $F\sd_\phi \langle t\rangle$ has no  peripheral  splitting over a cyclic or parabolic subgroup,  except the trivial one.

An unsatisfying aspect of this work is that I am unable to provide an algorithm certifying whether an element of $\Aut(F)$ is RH-noPS. But if it is, then we can do something.

\section{Conjugacy problems}
\subsection{Conjugacy of two relatively hyperbolic automorphisms without elementary splitting}
 
\begin{prop}\label{prop;RHnoES}
 Let $F$ be a free group. There is an (explicit) algorithm that, given two automorphisms, $\phi_1, \phi_2$, terminates if  both $\phi_i$ are RH-noPS, and  provides 

\begin{itemize} 
   \item either an  isomorphism $F\sd_{\phi_1} \langle t \rangle \to F\sd_{\phi_2} \langle t \rangle$   preserving fiber, and orientation;
   \item or  a certificate that     $F\sd_{\phi_1} \langle t \rangle$ and  $F\sd_{\phi_2} \langle t \rangle$ are not isomorphic by an isomorphism preserving fiber, orientation;
   \item or a non-trivial 
 peripheral splitting of either $F\sd_{\phi_i} \langle t\rangle$ over
 a cyclic subgroup, which is either maximal cyclic or
   parabolic. 
\end{itemize}

\end{prop}

The algorithm in question may terminate  even if one of the $\phi_i$ is not RH-noPS. It never lies though.

The following application is immediate, given Lemma \ref{lem;wellknown}. 
\begin{coro}
 The conjugacy problem for RH-noES elements of $\Out(F)$ is solvable: there is an algorithm that given two automorphisms that are  RH-noES, decides whether or not they are conjugated.
\end{coro}

Let us now prove  Proposition \ref{prop;RHnoES}
\begin{proof}

First, by Proposition \ref{prop;compute_RHS},  
we may
assume
 that we know explicitly both relative hyperbolic structures of $G_i =
 F\sd_{\phi_i} \langle t\rangle$ with presentations, as suspensions of
 subgroups of $F$,  of the
 parabolic subgroups (that are non-virtually cyclic). 
  Let us write $P_{1,j} = F_{1,j}
\sd\langle r_j\rangle, j=1, \dots, k$ conjugacy representatives of
maximal parabolic subgroups of $G_1$, with $P_{1,j}\cap F = F_{1,j}$
(recall that we have explicit presentations of the groups $P_{1,j}$ as
such suspensions).  Similarily, we have $P_{2,j} = F_{2,j}
\sd\langle r'_j\rangle, j=1, \dots, k'$ conjugacy representatives of
maximal parabolic subgroups of $G_2$. If $k\neq k'$, $G_1$ and $G_2$
cannot be isomorphic, hence we can assume that $k=k'$. 

In parallel, we then perform the three following searchs (so-called procedures, below). 

The first procedure is the enumeration of morphisms $G_1 \to G_2 \to G_1$. It stops when mutually inverse isomorphisms preserving orientation and sending the fiber into the fiber are found.

The second procedure is as follows. For incrementing integers $m$,  we compute
$N_{1,j}^{(m)}$ the intersection of all subgroups of $F_{1,j}$ of index
$\leq m$. Note that for each $j$, $N_{1,j}^{(m)}$  is a sequence of normal subgroups of
$P_{1,j}$, with trivial intersection (as $m$ goes to infinity). Denote
by    $\langle\langle \cup_{j} N_{1,j}^{(m)} \rangle \rangle$ the
normal closure in $G_1$ of their union (over $j$ for a fixed $m$). 
 Then we try to certify, using \cite{Pap}, that $\bar G_1^{(m)} = G_1/\langle\langle \cup_{j} N_{1,j}^{(m)} \rangle \rangle$ is hyperbolic.
 
We will denote by $K_{1}^{(m)} $ the kernel $\langle\langle \cup_{j}
N_{1,j}^{(m)} \rangle \rangle$ in $G_1$ of the previous quotient.

 Similarily, we compute  $\bar G_2^{(m)}$  and check that it is hyperbolic.
 Since $P_{1,j}/N_{1,j}^{(m)}$ is virtually cyclic, by virtue of the Dehn
 Filling theorem \cite[Thm. 1.1]{O_DF},    
 for $m$ large enough these
 groups are indeed hyperbolic. So this step of the second procedure 
 will eventually provide groups  $\bar G_i^{(m)}$, ($i=1,2$, $m$
 sufficiently large),  that are 
 certified hyperbolic. For all $m$,   $ \cup_j N_{1,j}^{(m)} $ is 
 contained in $F$ which is normal in $G_1$. Hence the whole group
 $K_{1}^{(m)} $ is contained in $F$, and $\bar G_1^{(m)}$ is naturally a
 suspension $\bar G_1^{(m)} = (F/     K_{1}^{(m)} ) \sd \langle \bar t
 \rangle$. The second procedure then calls the algorithm of Theorem
 \cite[3.2]{DCPout}  
 in order to decide whether  $\bar G_1^{(m)} $ and  $\bar G_2^{(m)} $ are isomorphic by a fiber and orientation preserving isomorphism. This is done in parallel for all incrementing $m$ for which the groups are certified hyperbolic. The second procedure stops if an integer $m$ is found so that   $\bar G_1^{(m)} $ and  $\bar G_2^{(m)} $ are not isomorphic by a fiber and orientation preserving isomorphism.

The third procedure is as follows. For both $i=1,2$, one enumerates
presentations of  $G_i = F\sd_{\phi_i} \langle t\rangle$ by Tietze
transformations, and,   for each one exhibiting a splitting of $G_i$ over
a cyclic subgroup as an amalgamation,  we check whether the splitting
is non-trivial (it suffices to check that both factors have a
generator that does not commute with the cyclic subgroup)   and,   for
each one of the form of an HNN extension, $\langle H, t \,|\, tct^{-1}
= c', \calR_H \rangle$     we check whether the stable letter $t$ is
non trivial (so the presentation is genuinely that of an HNN-extension
over a cyclic group). If we discover a non-trivial cyclic splitting,  
 we may check whether its cyclic edge subgroup is maximal cyclic or
 parabolic, using \cite[Thms 5.6 and  5.17]{Osin}. 
We then enumerate the conjugates of the parabolic subgroups, and if we
find that each parabolic subgroup has a conjugate contained in a
vertex group,  this third procedure stops, and outputs the splitting,
with the relevant conjugations of parabolic subgroups.

Now that we described the three procedures, we discuss the implication  of their termination.

If the first procedure terminates, then 
by  Lemma \ref{lem;wellknown},  there exists a fiber-and-orientation preserving isomorphism, and the two given automorphisms of $F$ are conjugated in $\Out (F)$. 
If the second procedure terminates, there cannot exist any  isomorphism $G_1 \to G_2$ preserving fiber and orientation (it would preserve the class of parabolic subgroups, characterised by being polynomial, and hence pass to the characteristic quotients). If the third procedure terminates, we have found a  non-trivial splitting of either $F\sd_{\phi_i} \langle t\rangle$ over a cyclic subgroup.

Now we need to show that there is always at least one procedure that terminates, {\it i.e.}  the following lemma.

\begin{lemma}\label{lem;call52}
Assume that $G_1$ and $G_2$ are  RH-noPS.    If the third procedure and the second procedure never terminate, then the first procedure terminates.

\end{lemma}

This Lemma is actually a consequence of a result obtained in a
collaboration of Vincent
Guirardel  and the author.

\begin{proof}  We assume that for all $m$ large enough, $\phi_m:  \bar
  G_1^{(m)} \to  \bar G_2^{(m)}$ is a fiber and orientation preserving
  isomorphism.  

Observe that eventually, $P_{1,j}/N_{1,j}^{(m)}$ contains a large
finite subgroup and an infinite order element normalizing it.  In
$G_2^{(m)}$ (for large $m$), all finite subgroups lie in  conjugates
of $P_{2,\ell}/N_{2,\ell}^{(m)}$ (by \cite[Lemma 4.3]{DG_CbDF}), and   since  
$G_2^{(m)}$ is eventually hyperbolic relative to this collection of
subgroups, any infinite order element normalizing a finite subgroup is in
the same conjugate of $P_{2,\ell}/N_{2,\ell}^{(m)}$. 
Thus, eventually each  $\phi_m$ must send, for each $j$, the group  $P_{1,j}/N_{1,j}^{(m)}$ on some conjugate
of some  $P_{2,\ell}/N_{2,\ell}^{(m)}$, ($\ell \leq k$). 
 Then, Theorem \cite[Thm. 5.2]{DG_CbDF}   
states that either $G_1$ or $G_2$  has a peripheral splitting over a
maximal cyclic, or a parabolic subgroup (which must be parabolic if we
assume that the third procedure does not terminate, hence in
contradiction with the assumption of the lemma), or there is an
isomorphism $\phi: G_1 \to G_2$ that commutes with infinitely many
$\phi_m$, up to composition with a conjugation in $\bar G_2^{(m)}$
(in other words, it makes a diagram    $ \begin{array}{ccc}  G_1 &
                                           \stackrel{\phi_\infty}{\to}
                                           & G_2
                                           \\  \downarrow & &
                                                              \downarrow
                                           \\  \bar G_1^{(m)}       &
                                           \stackrel{{\rm ad}_{\bar{g}}\circ \phi_m}{\to}
                                                                      &\bar G_2^{(m)} \end{array}
                                         $ commute, where $ {\rm
                                           ad}_{\bar{g}}$ is a conjugation). In this second circumstance,   $\phi$ has to preserve the fiber and orientation, since the $\phi_m$ do, and the kernels are co-final. 
\end{proof} 

\end{proof}

There is a slightly stronger version of Proposition \ref{prop;RHnoES} that we will need for the next part. 

Given a suspension $F\sd \langle t \rangle$, a transverse cyclic
peripheral structure is a tuple of elements of the form $ (t^{k_j}
f_j)_{j=1, \dots r} $, for $k_j \neq 0$ and $f_j \in F$.

A fiber-and-orientation preserving isomorphism between suspensions equipped with such structures is said to preserve the structure if it sends the conjugacy classes of the first exactly on the conjugacy classes of the second.

Let us amend our definition of RH-noPS, and say that a suspension with  a transverse cyclic
peripheral structure is RH-noPS relative to the transverse peripheral structure 
   if it is  properly hyperbolic
relative to a collection of polynomial subgroups of $F$, and the
suspension $F\sd_\phi \langle t\rangle$ has no non-trivial  
peripheral  splitting over
a subgroup of a parabolic subgroup, in which each element of the
transverse peripheral structure (which is a group) is conjugated to a factor.

The following Proposition is, as we said, similar to Proposition
\ref{prop;RHnoES}. The difference is in the presence of the transverse
cyclic peripheral structure (a minor difference) but also in the fact
that we had ambitionned to get the full list of fiber-and-orientation
preserving isomorphisms. This ambition is not realized unfortunately,
but enough is granted for the application in the next part. 

We keep
the notation $K_1^{(m)}, K_2^{(m)}$  for the normal subgroups
introduced in the proof of  Proposition
\ref{prop;RHnoES}.

\begin{prop}\label{prop;better}

 Let $F$ be a free group. There is an (explicit) algorithm that, given
 two automorphisms, $\phi_1, \phi_2$, and two transverse cyclic
 peripheral structures $\calP_1, \calP_2$ of $F\sd_{\phi_1} \langle
 t\rangle$  and   $F\sd_{\phi_2} \langle t\rangle$  respectively,
 terminates if  both $\phi_i$ are RH-noPS relative to their transverse
 peripheral structures, and  provides  
\begin{enumerate} 
   \item \label{point;1} either a  list of isomorphisms $F\sd_{\phi_1}
     \langle t \rangle \to F\sd_{\phi_2} \langle t \rangle$
     preserving fiber, orientation, and    transverse cyclic peripheral
     structure, and an integer $m$ such that 
for each $p\in \calP_2$, the centralizer of $\bar{p}$ in
$\overline{G_2}^{(m)}= G_2/K_2^{(m)}$ is the image of the centralizer of $p$
in $G_2$   
 and  such that, for  any other such isomorphism $\psi$, there is one, $\phi$, in the list, an element $g\in G_2$, such that for all $h\in G_1$, there is $z_h\in K_m$ for which $\psi(h)^g=\phi(h)z_h$.
   \item or  a certificate that     $F\sd_{\phi_1} \langle t \rangle$ and  $F\sd_{\phi_2} \langle t \rangle$ are not isomorphic by an isomorphism preserving fiber, orientation, and transverse cyclic peripheral structure;
   \item or a non-trivial peripheral splitting of either $F\sd_{\phi_i} \langle
     t\rangle$ over a cyclic subgroup, in which each element of the
     transverse peripheral structure is elliptic. 
\end{enumerate}
\end{prop}

The first point means that the list contains all isomorphisms up to conjugacy in $G_2$ and multiplication by a large element of $F$.

\begin{proof}
As in the proof of Proposition \ref{prop;RHnoES}, we use three procedures. 

The first procedure is the enumeration of morphisms $G_1 \to G_2 \to
G_1$. This procedure has an incrementing list $\calL$, which is empty
at the beginning. Every time  mutually inverse isomorphisms preserving
the transverse peripheral structure, the orientation and sending the fiber into the fiber are found, such that the isomorphism $G_1 \to G_2$ is not conjugated to any item of the list $\calL$,  the procedure stores  $G_1\to G_2$ into  $\calL$. We precise below when this first procedure is set to stop.

The second one slightly differs from Proposition \ref{prop;RHnoES}. We
still compute  $\bar G_1^{(m)}$ and $\bar G_2^{(m)}$ (and the images
of the transversal peripheral structure in them), and try to certify
that they are hyperbolic and that the assumption of point 1, on the centralizers, is
satisfied (which happens if $m$ is large enough, by Lemma
\ref{lem;in_the_DF}). Let us call this ``certification $\alpha$ for
$m$''. When this is the case, using \cite[Coro. 3.4]{DGu_gafa}, we
check whether these groups are rigid (in the sense that they have no
peripheral splitting over a virually cyclic group with infinite
center, in which each element of the transverse peripheral structure
is conjugated in a vertex group) and if they are we proceed and compute by
\cite[Coro. 3.5]{DGu_gafa} the complete list of isomorphisms  $\bar
G_1^{(m)} \to \bar G_2^{(m)}$ up to conjugacy in  $\bar
G_2^{(m)}$. Once this is done, we check which one of them are
fiber-and-orientation preserving, and preserve the transverse
peripheral structure,  and we record them in a list $\calL_m$.

The second procedure is set to stop if  an $m$ is found so that there is no fiber-and-orientation preserving isomorphisms  $\bar G_1^{(m)} \to \bar G_2^{(m)}$ that preserves the transversal peripheral structure.

The first procedure (which was run in parallel with the second) is set to stop   if a list $\calL$ is found and an integer $m$ 
 is found so that the following three conditions are satisfied. First,  ``certification $\alpha$'' for $m$ is done. 
Second,   $\calL_m$ is completely computed, and third, the currently computed list $\calL$ surjects on $\calL_m$, by the natural quotient map.

The third procedure  looks for a non-trivial peripheral splitting over a
cyclic subgroup  which is maximal cyclic or parabolic, 
in which each element of the transverse peripheral
structure is conjugated to a vertex group  (it is similar to that of   Proposition \ref{prop;RHnoES}). 

Observe that, if  the first procedure stops, we have in $\calL$ a list  of isomorphisms $F\sd_{\phi_1} \langle t \rangle \to F\sd_{\phi_2} \langle t \rangle$   preserving fiber (by Lemma \ref{lem;wellknown} (\ref{it_3} $\implies$ \ref{it_2})), orientation, and transverse cyclic peripheral structure, such that any other such isomorphism differs from one in the list by a conjugation, and the multiplication by elements in the fiber (in the sense of \ref{prop;better}-(\ref{point;1})). 
Observe also that if the second procedure stops, we have (as in \ref{prop;RHnoES}) a certificate that  $F\sd_{\phi_1} \langle t \rangle$ and  $F\sd_{\phi_2} \langle t \rangle$ are not isomorphic by an isomorphism preserving fiber, orientation, and transverse cyclic peripheral structure.

Again, we can conclude by the following lemma.
\begin{lemma}
Assume that $G_1$ and $G_2$ are  RH-noPS relative to their transverse
cyclic peripheral structures.    If the third procedure and the second
procedure never terminate, then the first procedure terminates ({\it
  i.e} there is a finite list of isomorphisms (preserving
the transverse peripheral structure, the orientation and sending the
fiber into the fiber) from $G_1$ to $G_2$, and
an integer $m$ as in (\ref{prop;better}-(\ref{point;1})) 
 such that any   isomorphism preserving fiber, orientation, and
 peripheral structure $\bar G_1^{(m)} \to \bar G_2^{(m)} $  is the
 image of an isomorphism  $G_1 \to G_2 $ in the list).
\end{lemma}

\begin{proof}

We endow $G_1$ and $G_2$ with the relatively hyperbolic structure
consisting of their parabolic subgroups, and for each hyperbolic
element in the transverse peripheral structure, the conjugates of the maximal cyclic
subgroup containing it. This  still makes a relatively hyperbolic
group (see \cite[1.7]{OsinIJAC} for instance). 
In this context, the assumption that the third procedure does not stops says that  the groups $G_1$ and
$G_2$ have no splitting which is peripheral (for this extended
peripheral structure),  over an elementary subgroup.

 Observe that, eventually, the groups $\bar
G_1^{(m)}$ and $\bar G_2^{(m)}$ are rigid (in the sense of the
second procedure). Indeed, if it is not the
case for the sequence $(\bar G_1^{(m)})$ for instance, then after
passing to a subsequence, for each $m$ in the subsequence, there are infinitely many automorphisms of $\bar G_1^{(m)}$  preserving the
transverse peripheral structure, the orientation and the fiber (namely
the iterates of a Dehn twist over an edge group of a splitting
falsifying rigidity).  From there,  by
\cite[Corolary 5.10]{DG_CbDF},  one gets that $G_1$ must have a splitting of a type contradicting
the previous paragraph.

 Therefore, for sufficiently large $m$, the second procedure can
compute the list $\mathcal{L}_m$ of isomorphisms preserving fiber
orientation, and peripheral 
structure  between $\bar G_1^{(m)}$ and 
$ \bar G_2^{(m)} $.   Since the second procedure does not
stop,  
  for all sufficiently large $m$, this list is not empty.

Assume that the first procedure does not stop.        Then,  for all
finite list $\calL$ of isomorphisms $G_1 \to G_2$ (preserving transverse
peripheral structure, fiber and orientation), and all $n$ there exists
$m>n$ and $\psi_{m}: \bar G_1^{(m)} \to \bar
G_2^{(m)} $,  an  isomorphism preserving fiber
orientation, and peripheral structure, that does not commute with any
element of the list $\calL$, even after conjugation by an element of
$\bar G_2^{(m)} $.  Choosing a sequence of lists $\calL$ exhausting the set
of  isomorphisms $G_1 \to G_2$ (with the prescribed preservation
property), and a increasing sequence of integers $n$, one gets a
sequence of $\psi_m$ as above, for a sequence of integers $m$ going to infinity.

 We may
use Theorem \cite[5.2]{DG_CbDF}   to extract a subsequence of
the isomorphisms $\psi_m$, and to find an isomorphism $\psi_\infty : G_1 \to G_2$ commuting with
the composition of $\psi_m$ with a conjugation (for all $m$ in the
extracted sequence), as our previous use of it in Lemma \ref{lem;call52}.  By cofinality, such a
$\psi_\infty$  has to preserve the fiber and the
orientation.  Since the transverse peripheral structure consists
exactly of the cyclic groups among  the parabolic groups of the
relatively hyperbolic structures under consideration, we deduce that
$\psi_\infty$ globally preserves the transverse cyclic
structure.  By \cite[Lemma 3.13]{DG_CbDF}, we know that, for
sufficiently large $m$, non conjugate parabolic subgroups in $G_2$ map
on non conjugate subgroups in $\bar G_2^{(m)}$.  
Since each 
$\psi_m$  preserves the transverse peripheral structure,  and (after
extraction of subsequence) commute with $\psi_\infty$,  it follows
that $\psi_\infty$ sends the transverse peripheral structures of
$G_1$  to that of $G_2$. Therefore  $\psi_\infty$ is eligible for
being recorded in the list of the first procedure, and thus eventually
appears in the sequence of our lists that served to define the
$\psi_m$. But that is a contradiction, by definition of the $\psi_m$.

\end{proof}

\end{proof}

\subsection{Conjugacy of two relatively hyperbolic automorphisms without parabolic splitting}

The  class of RH-noPS automorphisms, is larger than that of RH-noES. We can treat it as well, but it requires a little care.

Recall that if $F\sd_{\phi_1} \langle t \rangle$ splits as a
  graph of groups, then vertex groups and edge groups are suspension
  of subgroups of $F$ that are vertex groups and edge groups
  (respectively) of a graph of group decomposition of $F$ (see
  \cite[Lemma 2.6]{DCPout}).

\begin{prop}\label{prop;gdg_with_small_list}
 Let $\mathbb{X}_1, \mathbb{X}_2$ be graph-of-groups decompositions of $F\sd_{\phi_1} \langle t \rangle$ and $F\sd_{\phi_2} \langle t \rangle$ respectively, over cyclic subgroups.  

Assume that $\Phi_X : X_1 \to X_2$ is an isomorphism of the underlying
graphs of the decompositions, and that $m$ is an integer, and, for all vertex $v\in
X_1^{(0)}$,  $\calL_v$ is a list of isomorphisms from $\Gamma_{1,v}$ the vertex group
of $v$ in $\bbX_1$  to $\Gamma_{2,\Phi_X(v)} $ the group of $\Phi_X(v)$ in
$\bbX_2$, that are all  as in (\ref{prop;better}-(\ref{point;1}))
(for  $G_i=\Gamma_{i,v}$) for 
the transverse peripheral structure consisting of the adjacent edge subgroups,
and the fiber and orientation defined by \cite[Lemma 2.6]{DCPout}.

The following are equivalent.

\begin{enumerate}
\item There is an isomorphism $\Phi: \pi_1(\bbX_1, \tau) \to \pi_1(\bbX_2, \Phi_X(\tau))$ that is fiber and orientation preserving, that induces a graph-of-groups isomorphism, and  that induces $\Phi_X$ at the level of graphs.

\item \label{point;2} There is $\Phi_0=(\Phi_X, (\phi_v), (\phi_e), (\gamma_e))$ an isomorphism of graphs of groups, 
such that $\phi_v$   is fiber-and-orientation preserving, and 
 such that the linear diophantine equation (\ref{eq;Dioph_lin}) for
 the   $\gamma_i$ being the
images of a fixed basis of $F$, and of $t$, by $\Phi_0$,    has a
 solution.

\item \label{point;3} There is $\Phi'_0=(\Phi_X, (\phi_v'), (\phi_e),
  (\gamma_e'))$ an isomorphism of graphs of groups, such that
  $\phi'_v\in \calL_v$, and such that the linear diophantine equation
  (\ref{eq;Dioph_lin})  for the $\gamma_i$ being the
images of a fixed basis of $F$, and of $t$, by $\Phi'_0$,   has a solution. 

 \end{enumerate}
\end{prop}

\begin{proof}

The first point implies the second: by Lemma \cite[2.6]{DCPout},
  all vertex groups are suspensions of their intersections with $F$, therefore if $\Phi$ preserves the fiber and orientation, so do all $\phi_v$, and the equation (\ref{eq;Dioph_lin}) admits an obvious solution (the null solution).

The second point implies the first, because, by Proposition
\ref{prop;orbit_fiber},  the system (\ref{eq;Dioph_lin}) has a
solution if and only if there is a modular 
graph-of-group  automorphism of $\pi_1(\bbX_2, \Phi_X(\tau))$  that sends $\Phi_0(F)$ exactly on $F$, and preserves orientation.

The third point obviously implies the second one. 

We need to show that the second point implies the third one. This is
a more subtle part. It basically says that if there is an isomorphism
of graph of groups whose orbit under the small modular group
intersects the set of fiber-preserving isomorphisms, then there is one
that is accessible to us, which might not  quite
be in the same orbit for the action of the small modular group, but
that share the property that its orbit intersects  the set of
fiber-preserving isomorphisms. The key is in the assertions, in Lemma \ref{lemma;use_dehn} that the
different products are in $F$, and this is ensured by the ({\it a
  priori} non immediate) assumption
(\ref{prop;better}-(\ref{point;1})).

  So,  we have $\Phi_0=(\Phi_X, (\phi_v), (\phi_e), (\gamma_e))$  an
   isomorphism of graphs of groups, as in the second point.

By assumption on $\calL_v$, the given isomorphisms $\phi_v$ differ
from isomorphisms in $\calL_v$ by conjugation (in the target), and
multiplication by elements of $F$ that are also in the vertex group $G_v$. By composing with inert twists (vanishing in $\Out(F\sd_{\phi_2} \langle t \rangle)$), we can assume that each $\phi_v$ has same image as an element of   $\calL_v$ in the  Dehn Filling reduction $\overline{\Gamma_v}^{(m)} \to\overline{\Gamma_{\phi_X(v)}}^{(m)}$.

\begin{lemma}\label{lemma;use_dehn}
 There is an isomorphism of graphs of groups $$\Phi'_0 = (\Phi_X,
 (\phi'_v), (\phi_e), (\gamma'_e))$$ such that $\phi'_v \in \calL_v$
 has same image as $\phi_v$ in the Dehn Filling reduction
 $\overline{\Gamma_v}^{(m)} \to\overline{\Gamma_{\phi_X(v)}}^{(m)}$,
 and such that for all edge $e$,  $(\gamma'_e)^{-1}  \gamma_e \in F $,
 and also  $ \gamma_e (\gamma'_e)^{-1} \in F $, and  $(\gamma'_e)  \gamma_e^{-1} \in F $.
\end{lemma}  

\begin{proof}
 We thus construct  $\Phi'_0$ using these elements $\phi'_v \in \calL_v$. The morphisms $\phi_e$ are given by the marking of the cyclic edge groups. We need that there exists elements $\gamma'_e$ completing the collection into an isomorphism of graph-of-groups, but this is actually the condition that, for $v=o(e)$,  $\phi'_{v}$ preserve the peripheral structure 
  of the adjacent cyclic edge groups. 
Note that one can choose the $\gamma'_e$ up to a multiplication on the left by an element of $\Gamma_{o(e)}$ centralising $i_e(\Gamma_e)$. 

Once such elements $\gamma_e'$ are chosen,  $\Phi'_0$ is defined. Recall that on a Bass generator  $e\in X^{(1)} \setminus \tau$,  $\Phi'_0 (e) =  (\gamma'_{\bar e})^{-1}   \Phi_X(e) (\gamma'_e)$. 
We need to compute how $\Phi'_0$ differs from $\Phi_0$ on Bass generators. 

Let us call $c_e$ the marked generator of the edge group $\Gamma_e$. To make notations readable, we will still write $c_e$ for $i_{e}(c_e)$. 

One has $c_{\Phi_X(e)} = \phi_v (c_e)^{\gamma_e}$ by Bass diagram
\ref{BassDiagram}. 
By virtue of $\phi_v'$ preserving the peripheral structure (for
$v=o(e)$) this is also $=  (\phi'_v (c_e))^{h_e}  )^{\gamma_e}$, for
some $h_e \in \Gamma_{\phi_X(v)}$ which can be chosen up to left
multiplication by an element centralising $\phi'_v (c_e)$. By Bass
diagram \ref{BassDiagram}  (for $\phi'_v$) this is $=  ((c_{\Phi_X(e)})^{(\gamma'_e)^{-1}} )^{h_e}  )^{\gamma_e}  $. It follows that 
$(\gamma'_e)^{-1} h_e \gamma_e$ centralises $c_{\Phi_X(e)}$, and lies in $\Gamma_{\phi_X(v)}$, for $v=o(e)$.

Recall that  $\gamma'_e$ can be chosen  up to a multiplication on the left by an element of $\Gamma_{o(e)}$ centralising $i_e(\Gamma_e)$. By a right choice of the collection of $\gamma'_e$ (or, in different words, by the right application of Dehn twists), we may assume that $(\gamma'_e)^{-1} h_e \gamma_e= 1$.

By virtue of $\phi'_v$ coinciding with $\phi_v$ in the Dehn filling
$\overline{\Gamma_v} ^{(m)}$, the image $\bar h_e$ of $h_e$  actually
centralises $\overline{\phi'_v(c_e)}$. By assumption on $m$ (see
\ref{prop;better} -(\ref{point;1})), the
centralizers of the transverse peripheral structure in $\bar{G_2}^{m}$
are the images of the centralizers in $G_2$, so this makes $h_e = z_e f_e$ for $z_e$ centralising   $\phi'_v (c_e)$ and $f_e$ in $F$.
Thus, we may choose it so that $z_e=1$, hence $h_e\in F$. Finally,
since $F$ is normal, and  $(\gamma'_e)^{-1} h_e \gamma_e= 1$, it
follows that   $(\gamma'_e)^{-1}  \gamma_e \in F$.  The two other
relations are obtained respectively by conjugating  by $\gamma_e^{-1}$
($F$ is normal) and taking the inverse of the later.
\end{proof}

We can resume the proof of the Proposition (second point implies the third one).
Finally, we need to check that the automorphism $\Phi'_0$ provided by
the Lemma is suitable. We can compare the images of the Bass generator
$e$, namely $\Phi_0(e)$ to $\Phi_0'(e)$.
\begin{lemma} For all edge $e$,  $\Phi_0'(e)^{-1} \Phi_0(e)  \in F$,
  and for all edge $e'$, the number $n(\Phi_0'(e)^{-1} \Phi_0(e), e')$
  that counts the number of occurences of $e'$ in the normal form of    
                                 $\Phi_0'(e)^{-1} \Phi_0(e) $, minus
                                 the number of occurences of
                                 $\overline{ e'}$ (see \S \ref{para;recall}), is $0$. 
\end{lemma} 
\begin{proof}
 Recall that  $\Phi_0(e) = \gamma_{\bar e}^{-1} \phi_X(e) \gamma_e$ and 
 $\Phi_0'(e) = (\gamma'_{\bar e})^{-1} \phi_X(e) \gamma'_e$. The
 difference is therefore  
$$\Phi_0'(e)^{-1} \Phi_0(e)  =   (\gamma'_e)^{-1}  \phi_X(e)^{-1}
(\gamma'_{\bar e}  \gamma_{\bar e}^{-1})    \phi_X(e)     \gamma_e  $$ 
which can also be written 
$\Phi_0'(e)^{-1} \Phi_0(e)  = 
  (\gamma'_e)^{-1}      (\gamma'_{\bar e}  \gamma_{\bar  e}^{-1})^{
    \phi_X(e)}    \gamma_e   $ and slightly less naturally, 
$\Phi_0'(e)^{-1} \Phi_0(e)  = (  \gamma_e (\gamma'_e)^{-1} (
                                 \gamma'_{\bar e} \gamma_{\bar e
                                 }^{-1} )^{\phi_X(e )} )^{\gamma_e}$.
 
  Since we established that  $\gamma_e (\gamma'_e)^{-1} \in F$, 
$\gamma'_{\bar e} \gamma_{\bar e }^{-1} \in F$ and $F$ is normal, $\Phi_0'(e)^{-1} \Phi_0(e) \in F$.

Moreover, since all factors in the product $(  \gamma_e (\gamma'_e)^{-1} (
                                 \gamma'_{\bar e} \gamma_{\bar e
                                 }^{-1} )^{\phi_X(e )} )^{\gamma_e}$
                                 are in vertex groups, except
                                 $\phi_X(e )$ and $\phi_X(e )^{-1}$
                                 (which both appear once). Therefore,
                                 for all edge $e'$, the quantity
                                 $n(\gamma'_{\bar e} \gamma_{\bar e
                                 }^{-1}, e') $  is $0$.
\end{proof}

We may now finish the argument and show that the system of equations
(\ref{eq;Dioph_lin})  for the $\gamma_i$ being the images by
$\Phi_0'$ of a basis of $F$, and of $t$,  has a solution.

Consider an element $f$ of the given basis of $F$. Write the normal
form in $\pi_1(\bbX_1, \tau)$ as $f = g_0e_1g_1e_2 \dots
e_ng_{n+1}$. The normal form of its image  by $\Phi_0$ in
$\pi_1(\bbX_2, \Phi_X(\tau)) $ is thus 
$$\Phi_0(f) = \phi_{v_0} (g_0) \, \Phi_0(e_0) \, \dots  \, \Phi_0(e_n) \,   \phi_{v_{n+1}} (g_{n+1}).  $$

By  assumption
(\ref{prop;better}-(\ref{point;1})), each $\phi_{v_i}(g_i)$ differs from $\phi'_{v_i}$ by a conjugation
(say by $ \xi_i$),
and the multiplication on the right by an element (say $f_{i,\ell}$)
of $F$, that lies in a vertex group (hence $n(f_{i,\ell}, e') = 0$ for
all $e'$). 
We also established in the lemma that $\Phi_0'(e_i)^{-1} \Phi_0(e_i)  \in F$ (call
it $f_{i,r}$), and that $n(f_{i,r}, e') = 0$ for
all $e'$).

We thus get  
$$\Phi_0(f) = (\phi'_{v_0} (g_0))^{\xi_0}   \,  f_{0, \ell} \,
\Phi'_0(e_0) f_{0, r} \, \dots  $$ $$ \hfill \dots (\phi'_{v_{n}} (g_{n}))^{\xi_n} \,
f_{n, \ell} \, \Phi'_0(e_n) \, f_{n,r}   \, (\phi'_{v_{n+1}}
(g_{n+1}))^{\xi_{n+1}}\, 
f_{n+1, \ell}.$$

We first observe that for all $e$, $n( \Phi_0(f) ,e ) = n(\Phi'_0(f)
,e)$ since all $n(f_{i,\ell}, e) =0$ and all $n(f_{i,r}, e)=0$.

In $H_1 ( F\sd_{\phi_2} \langle t \rangle)$ this normal form turns
into $$\overline{\Phi_0(f)} = \left( \prod_{i=0}^{n+1} \overline{ \phi'_{v_i}
  (g_i) } \right)  \times \left( \prod_{i=0}^{n}   \overline{ \Phi'_0(e_i)}
\right)   \times (f_{tot}),$$ where $f_{tot}$ is  the image of
$\prod_i f_{i,\ell} f_{i,r}$ in  $H_1$, hence is in the image of $F$.

Of course, $$\overline{\Phi'_0(f)} = \left( \prod_{i=0}^{n+1} \overline{ \phi'_{v_i}
  (g_i) } \right)  \times \left( \prod_{i=0}^{n}   \overline{ \Phi'_0(e_i)}
\right). $$

In the notations of equation (\ref{eq;Dioph_lin}), $\delta (
\overline{\Phi'_0(f)}) = \delta(\overline{\Phi_0(f)})$ because $\delta(f_{tot}
)= 0$ as $f_{tot}$ is in $F$. Moreover, we noticed that for all edge $e$, $n(\Phi'_0(f), e) =
n(\Phi_0(f), e)   $. Therefore the two
systems of Diophantine
equations   (\ref{eq;Dioph_lin})  for the $\gamma_i$ being the
images of a fixed basis of $F$ by $\Phi_0$ and  for the $\gamma_i$
being the
images of a fixed basis of $F$ by $\Phi'_0$, are syntaxically the same system of
equations.
Tautologically,  if one has a
solution, the other aslo.

\end{proof}

For the last (and main) result, we'll need the theory of JSJ
decompositions for relatively hyperbolic groups. The theory initiated
by Rips and Sela is developed by Guirardel and Levitt in a very stable
and useful formulation. We refer to \cite{GL10, GL_JSJ}, from which we
recall 
the following existence and characteristic result. 

\begin{prop}\cite[Thm. 13.1, Coro. 13.2]{GL10} \label{prop_GL_JSJ}
Let $(G,\calP)$ be a torsion free relatively freely indecomposable,
relatively hyperbolic group. The canonical JSJ splitting of
$(G,\calP)$ is a finite graph-of-groups decomposition of $(G,\calP)$
with edge groups elementary  (cyclic or parabolic), bipartite, such that the
groups of a vertex of one color (black) are elementary, and those of
the other color (white) are either fundamental groups of surfaces with
boundary, the adjacent edge groups being associated to the boundary
components, or groups that inherit a relatively hyperbolic structure
for which there is no  non-trivial    
elementary splitting in which the      adjacent edge groups are
conjugated into factors    (the later are called rigid).

The canonical JSJ splitting is such that any automorphism of
$(G,\calP)$ induces an automorphism of the splitting, and is also such
that any other elementary splitting of $(G,\calP)$ has a common
refinement with it. 
\end{prop} 

Here, a  refinement is an equivariant blow-up of vertices in the
Bass-Serre tree. 

Recall that in the case of a suspension of a finitely generated group, no vertex group
can be a surface group (see \cite[2.11]{DCPout}). Moreover, in the
case of a suspension of a free group by an automorphism that is
RH-noPS, then all black vertex groups are cyclic, non-parabolic, by
definition of RH-noPS.

\begin{prop} Given a finitely generated free group and an automorphism
  $\phi$ of $F$ that is RH-noPS, one can compute the canonical JSJ
  splitting of $F\sd_\phi \bbZ$. 
\end{prop}

\begin{proof}
Let $(G, \calP)$ be the relatively hyperbolic structure for $F\sd_\phi
\bbZ$. The computation of the relative hyperbolicity structure was done in
\ref{prop;compute_RHS}.  We can thus enumerate the non-trivial bipartite
peripheral cyclic splitting of $(G, \calP)$. 
 Assume that we are proposed   such a  bipartite  cyclic  splitting
of $(G, \calP)$. We need to decide whether or not this proposed
splitting is the JSJ splitting.   Recall that vertex groups are
suspensions of finitely generated subgroups of $F$, and 
 we may find a presentation of the
semidirect product for each of them, by enumeration. We use a result
of Touikan here, specifically, \cite[Thm C.]{Tou} that 
allows to check whether or not,  the white vertex groups  are rigid 
                        (in the sense of Proposition \ref{prop_GL_JSJ}).
 Note that since we assume the absence
of parabolic splitting, we do not need to satisfy the assumption of
\cite{Tou} on algorithmic tractabiliy of parabolic subgroups. 
One can also easily check whether the black vertex
groups are cyclic, and edge groups are maximal cyclic in their white
adjacent vertex group.  If all these conditions are verified for a
splitting $\mathbb{X}$, the
splitting has, by the previous proposition, a common refinement with the JSJ
splitting $\mathbb{X}_0$. However, because of lack of surface groups
in  $\mathbb{X}$ and  $\mathbb{X}_0$, for all vertex
group $G_v$ in either of them, equipped with the peripheral structure
of its adjacent edge groups,  all splittings
of $G_v$ over cyclic groups in which the peripheral structure is
elliptic, are trivial.

 Thus
considering $\mathbb{T}$ the Bass-Serre tree of a common refinement of
$\mathbb{X}$ and  $\mathbb{X}_0$, it is apparant that the edges to be
collaped to obtain the tree of $\mathbb{X}$ and the tree of
$\mathbb{X}_0$ are exactly those with an end of valence $1$, and after
them, those with an end of valence $2$ with same stabilizer ({\it
  i.e. } the redundant vertices; note that there is a choice, but
either choice lead to the same tree). 
 Since the two collapses toward $\mathbb{X}$ and $\mathbb{X}_0$ can be
 made by the same choices od edges to collapse, it follows that
 $\mathbb{X}$ and $\mathbb{X}_0$ coincide.

\end{proof}

\begin{prop} \label{prop;last_main}
  Let $F$ be a finitely generated  free group. There is an (explicit) algorithm that given two automorphisms $\phi_1, \phi_2$ of $F$ terminates if both are  RH-noPS, and indicates whether they are conjugated in $\Out(F)$. 

\end{prop}

\begin{proof}

One can compute the JSJ decomposition of both suspensions, by the
previous proposition. 
The situation reduces to the case where an isomorphism of underlying graphs of the JSJ  splittings is chosen, and we need to decide whether there is an isomorphism of graph of groups (inducing that isomorphism of underlying graphs) that preserve the fiber, and the orientation. We also choose a maximal subtree $\tau$ of the underlying graph $X$ so that all edges outside this subtree correspond to Bass generators in the fundamental group of the graph of group.

Let us write $\bbX_1, \bbX_2$ the   JSJ  splittings of $G_1= F\sd_{\phi_1} \langle t \rangle $ and $G_2 =F\sd_{\phi_2} \langle t \rangle $, and we identify the underlying graphs, according to the choice of isomorphism above (the algorithm has to treat all possible such isomorphisms of graphs in parallel).  We write $\Gamma_{v,i}$ for the vertex group of $v$ in $\bbX_i$.

 We remark that any elementary vertex group (or edge group) is cyclic,
 and transveral to the fiber (Lemma \cite[2.7]{DCPout}), hence the
 orientation of the suspension provides a canonical marking of each
 edge group.  For each non-elementary vertex group, Lemma
 \cite[2.6]{DCPout} indicates that it is a suspension, and it is  equiped with the
 thus marked cyclic transverse peripheral structure of its adjacent
 edge groups.    Since, by \cite[2.11]{DCPout}, it is not a surface nor a free group,   
 by property of the JSJ decomposition, it is  RH-noPS
relative to the peripheral structure, and  it is possible, thanks 
to Proposition \ref{prop;better} to decide whether there is an isomorphism $\Gamma_{v,1} \to \Gamma_{v,2}$ , preserving fiber, orientation, and the (cyclic, transverse) peripheral structure of its adjacent edge groups, thus marked.

If for some vertex there is no such isomorphism, then there cannot be
any fiber-and-orientation preserving isomorphism between $G_1$ and
$G_2$, inducing this graph isomorphism (hence $\phi_1$ and $\phi_2$
are not conjugated, in view of Lemma \ref{lem;wellknown}). If, on the contrary, for all such vertices, there exists such an isomorphism, then  Proposition \ref{prop;better}  actually provides  a list $\calL_v$ of such isomorphisms, that satisfies  (\ref{prop;better}-(\ref{point;1})), for all vertex $v \in X^{(0)}$. Then,  for any choice $(\Phi_v, v\in X^{(0)})$ in $\prod_{v} \calL_v$,  one can extend this collection into an isomorphism of graph-of-groups $\Phi$, by chosing appropriately the images of the Bass generators to be $b_{e}, e\in X \setminus \tau$ (so that they conjugate the edge group in their origin vertex to the edge dgroup in their target vertex). Let us write $(\Phi_s, s\in \prod_{v} \calL_v)$ the collection thus obtained.

Using \ref{prop;orbit_fiber}  
 we can decide whether, given $\Phi_s$ for some $s\in \prod_{v} \calL_v$,  there is an automorphism of graph-of-groups, in the orbit of $\Phi_s$ under the small modular group of $G_2$, that is  fiber-and-orientation preserving.
If there is one, then we may stop and declare, in view of Lemma \ref{lem;wellknown} that $\phi_1$ and $\phi_2$ are conjugated.

Assume then that  there is none such  fiber-and-orientation preserving
automorphism in the orbits of all $\Phi_s, \, s\in \prod_{v} \calL_v$. By Proposition \ref{prop;gdg_with_small_list} (``\ref{point;2}$\implies$\ref{point;3}''), there is no  isomorphism of graph-of-groups preserving fiber and orientation. We are done. \end{proof}

\subsection{A lemma on Dehn fillings}

In this paragraph we prove a Lemma that we used above. We refer the
reader to the setting of \cite[\S 7.]{DGO}. 

In particular we will use
the parabolic cone-off construction, which is a specific way to cone
off horosheres of a system of horoballs of a space associated to a
relatively hyperbolic group. Using this specific way allows to get
quantitative hyperbolicity estimates, and rotating families, while preserving, almost
without distortion, most of what
occurs (locally) in a thick part of that space.

\begin{lemma} \label{lem;in_the_DF}  
Let $(G,\calP)$ be a relatively hyperbolic group, and $\gamma \in G$, a hyperbolic element.  

There exists $m_0$ such that for all $m>m_0$, if  $h$ is such that $\bar h$ centralises $\bar \gamma$ in $\bar{G}^{(m)}$, then, there is $z\in K_m$, and $h'$ centralizing $\gamma$ such that $h=h'z$.

\end{lemma}

The bound on $m$ will be explicit (but we do not need this particular aspect), though  probably not optimal.

\begin{proof}

Consider a hyperbolic space $X$ associated to
$(G,\calP)$,  upon which $G$ acts as a geometrically finite group, and
for convenience, let us choose it to be a cusped-space as defined by
Groves and Manning (see \cite[\S3]{GM}). Let $\delta$ be a
hyperbolicity constant for $X$. 

 After rescaling, we may assume that
the hyperbolicity constant is actually less than a specific constant $\delta_c$ furnished
by \cite[3.5.2]{Cou_GGD}, (that will allow, as
we did in  \cite[\S 7]{DGO}     to satisfy \cite[5.38]{DGO}, to
 ensure quantitatively that the parabolic
cone-off construction over a separated system of horoballs is
hyperbolic).

 By assumption  $\gamma$ is hyperbolic in  $X$, so let $\rho_0$ be a quasi-axis,
 and $\|\gamma\|$ the translation lenght of $\gamma$ on this axis.  We choose
 $m_0$ such that  $N_{j,m_0}\setminus\{1\}$ does not intersect the
 ball of $P_j$ of radius $10\sinh(r_U) \times 2^{100\delta +
   \|\gamma\|}$, for $r_U$ fixed as in \cite[\S 5.3]{DGO}, namely
 $r_U=5\times 10^{12}$. 

We choose $L_0$ such that in a $\delta$-hyperbolic space, all
$L_0$-local geodesic is a quasigeodesic (for some constants). And we
choose $L_1$ so that any quasigeodesic with these constants is at distance $\leq L_1$
from a geodesic with same end points. The constant $L_3$ is set to be
$\geq 50\times 16\times 900$.

  Consider  $\calH_0$ the $2$-separated invariant system of horoballs
  of $X$ (associated to $\calP$),   and in this system,  consider 
  the system of horoballs at  
   depth $(50\delta+ 20\|\gamma\| +L_0+L_1+L_3)$,  which we call $\calH$.  
This way, $\rho_0$, which has a  fundamental domain for $\gamma$ of length
$\|\gamma\|$, and which intersects $X\setminus \calH_0$,   does not  get
$(L_0+L_1+L_3)$-close   to an horoball of  $\calH$.

We then consider the parabolic cone-off $\mathcal{C}(X, \calH)$, as defined in
\cite[\S 7, Def. 7.2]{DGO}, for this pair, and for a radius of cone $
r_U=5\times 10^{12} $.

By  \cite[Lem. 7.4]{DGO},  this parabolic cone-off is
$\delta_p$-hyperbolic, (for a value of $\delta_p$ estimated to be
$16\times 900$
in \cite[\S 5.3]{DGO}) and moreover the image of $\rho_0$ is a
$L_0$-local geodesic. By our choices of constants it follows that  a
quasi-axis $\rho$ of $\gamma$ in the parabolic cone-off does not approach any
of the cones by a distance  $50\delta_p$.

We now work only in the parabolic cone-off.

Observe also that,   for the chosen $m$, any (non-trivial) element of
$N_{j,m}$ translate on the corresponding horosphere of $\calH$ by a
distance of at least  $10\sinh(r_U)$ (measured in the graph distance
of the horosphere). Therefore,  for the chosen $m$, $K_m$ is the group
of a very rotating family at the apices of the parabolic cone-off, in
the sense of \cite[\S 5.1]{DGO}.

Let $x_0 \in \rho$. 
  The segment $[x_0, \gamma x_0]$ is contained in $\rho$, and thus does not  get $50\delta_p$ close to an apex.

Consider the segment  $[x_0, hx_0]$ and for all apex $a$ on it, define the two points $a_-$ and $a_+$ on  $[x_0, a]$ and $[a, hx_0]$ (subsegments of $[x_0, hx_0]$) at distance  $27\delta_p $ from $a$ (they exist since the cones have much larger radius than $27\delta_p$,  and $x_0$ is not in a cone).

   By multiplying $h$ by elements of $K_m$, we may assume that $d(a_-, a_+) = d( a_-,   (\Fix(a)\cap K_m)  a_+    )$, for every apex $a$ in the segment $[x_0, hx_0]$.

By assumption,  $h\gamma h^{-1}\gamma^{-1} \in K_m$, and we can assume that it is
non trivial (otherwise there is nothing to prove). By   \cite[Lemma
5.10 (pointed Greendlinger lemma)]{DGO},  the segment $[  h\gamma x_0,
\gamma hx_0 ] $ contains an apex $a_0$ and a $5\delta_p$-shortening pair: a
pair of points at distance $27\delta_p$ from the same apex, such that
the image of one by an element fixing the apex and in the rotating
group is at distance $\leq 5\delta_p$ from the other.

Hyperbolicity in the pentagon $(x_0,  hx_0, h\gamma x_0, \gamma h x_0,
\gamma x_0)$, together
with the absence of apices in  $[x_0, \gamma x_0]$, and in its image after
translation by $h$, $[hx_0, h\gamma x_0 ]$, shows that at least one
segment among $[x_0,  hx_0]$   and $[\gamma x_0, \gamma hx_0]$ must
get at distance $\delta_p$ from $a_0$ and $5\delta_p$-follows two arcs of $[a_0,
\gamma h
x_0]$ and $[a_0, h\gamma x_0]$ (both subsegments of $[ h\gamma
x_0, \gamma h x_0]$) for at least $50\delta_p$. By properties of
rotating families, we see that one of the two segments $[x_0,  hx_0], [\gamma x_0, \gamma hx_0]$   must contain the apex $a_0$. Assume that only one of them contains $a_0$, and let us say that it is  $[x_0,  hx_0]$ (if it is the other, the argument is identical).  
Hyperbolicity forces the  $5\delta_p$-shortening pair of  $[  h\gamma x_0,
\gamma hx_0 ] $ at $a_0$ to be $5\delta_p$-close to $a_-$ and $a_+$. Therefore some element of $K_m\cap \Fix(a_0)$ takes $a_+$ to a point at distance at most $20\delta_p$ from $a_-$, thus contradicting the above minimality condition. 

It follows that both  $[x_0,  hx_0]$   and $[\gamma x_0, \gamma hx_0]$ (which is its image by $\gamma$)  must contain $a_0$.

But then,  the image of $a_0$ by $\gamma$ is at distance at most $\|\gamma\|$ from $a_0$. By separation of apices, it must then be $a_0$, and $\gamma$ fixes an apex, contrarily to our assumption.

\end{proof}

{\footnotesize

\vskip 1cm

\noindent {\sc Fran\c{c}ois Dahmani, \\ Universit\'e Grenoble Alpes, Institut Fourier, F-38000 Grenoble, France.
\\{\tt e-mail: francois.dahmani@univ-grenoble-alpes.fr} 

}}

\begin{thebibliography}{99}

\bibitem[ALM]{Arzh} G. Arzhantseva,  J.-F. Lafont,  A. Minasyan, 
Isomorphism versus commensurability for a class of finitely presented
groups. J. Group Theory 17 (2014), no. 2, 361-378

\bibitem[BMV]{BMV} O. Bogopolski, A. Martino, E. Ventura, Orbit decidability and the conjugacy problem for some extensions of groups. 
Trans. Am. Math. Soc. 362, No. 4, 2003-2036 (2010). 
 

\bibitem[Ba]{Bass} H. Bass, Covering theory for graph of groups, J. Pure
  Appl. Algebra 89 (1993), 3--47.


 
\bibitem[Br1]{Br_gafa} P. Brinkmann,  Hyperbolic automorphisms of free
  groups. Geom. Funct. Anal. 10 (2000), no. 5, 1071--1089. 

\bibitem[Br2]{Br_split}  P. Brinkmann,  Splittings of mapping tori of free
  group automorphisms. Geom. Dedicata 93 (2002), 191--203.   

\bibitem[D]{D_israel} F. Dahmani, Existential questions in (relatively)
  hyperbolic groups. Israel J. Math. 173 (2009), 91--124.  

\bibitem[DGu1]{DGu_JoT}  F. Dahmani, V. Guirardel, Foliations for solving
  equations in groups: free, virtually free and hyperbolic groups,
  J. of Topology, 3, no. 2 (2010) 343--404. 

\bibitem[DGu2]{DGu_gafa} F. Dahmani, V. Guirardel,  The isomorphism problem
  for all hyperbolic groups. Geom. Funct. Anal. 21 (2011), no. 2,
  223--300. 


 





\bibitem[DGr]{DGr} F. Dahmani, D.  Groves,  The isomorphism problem for
  toral relatively hyperbolic groups. Publ. Math. Inst. Hautes
  \'Etudes Sci. No. 107 (2008), 211--290. 









\bibitem[Lo]{Los} J. Los, On the conjugacy problem for automorphisms of free groups.
Topology 35 (1996), no. 3, 779--808. 



\bibitem[Lu1]{Lu1}
M.~Lustig.
 Structure and conjugacy for
automorphisms of free groups I.
MPI-Bonn preprint series 2000, No.
{241}; http://www.mpim-bonn.mpg.de/preprints

\bibitem[Lu2]{Lu2}
M.~Lustig.
 Structure and conjugacy for
automorphisms of free groups II.
MPI-Bonn preprint series 2001, No.
{4}; http://www.mpim-bonn.mpg.de/preprints

\bibitem[Lu3]{Lu3}
M.~Lustig,
 Conjugacy and centralizers for iwip automorphisms of free
groups, in ``Geometric Group Theory'', Trends in Mathematics,
197--224. Birkh\"auser Verlag, Basel, 2007



\bibitem[P]{Pap} P. Papasoglu, An algorithm detecting hyperbolicity. Geometric and computational perspectives on infinite groups (Minneapolis, MN and New Brunswick, NJ, 1994), 193--200, DIMACS Ser. Discrete Math. Theoret. Comput. Sci., 25, Amer. Math. Soc., Providence, RI, 1996.

\bibitem[Sel]{Sel} Z. Sela,  The isomorphism problem for hyperbolic groups,
  Ann. Math., 141 (1995), 217--283.

\bibitem[Ser]{Serre} J.-P. Serre, Arbres, Amalgames, $SL_2$, Soci\'et\'e
  Math\'ematique de France,  Asterisque no. 46 (1977). 

\bibitem[Sta]{Sta} J. Stallings, Group theory and three-dimensional
  manifolds. Yale Math., Monographs.  no.4, Yale University Press (1971). 


\end{thebibliography}

\begin{thebibliography}{99}



\bibitem[BFH-00]{BFH1_Tits1} M. Bestvina, M. Feighn, M. Handel, The Tits alternative for Out(Fn). I. Dynamics of exponentially-growing automorphisms. Ann. of Math. (2) 151 (2000), no. 2, 517--623. 

\bibitem[BFH-05]{BFH2_Tits2} The Tits alternative for Out(Fn). II. A Kolchin type theorem. Ann. of Math. (2) 161 (2005), no. 1, 1--59. 
 
\bibitem[B]{Br_gafa} P. Brinkmann,  Hyperbolic automorphisms of free
  groups. Geom. Funct. Anal. 10 (2000), no. 5, 1071--1089. 



\bibitem[C]{Cou_GGD} R. Coulon, Asphericity and small cancellation
  theory for rotation families of groups, Groups Geom. Dyn. {\bf 5},
  (2011), no.4  729--765.

\bibitem[DG-11]{DGu_gafa} F. Dahmani, V. Guirardel,  The isomorphism problem
  for all hyperbolic groups. Geom. Funct. Anal. 21 (2011), no. 2,
  223--300. 

\bibitem[DG-13]{DG_PPS}  F. Dahmani, V. Guirardel, Presenting parabolic subgroups, Alg. Geom. Top. 13 (2013) 3203-3222. 
 
\bibitem[DG-15]{DG_CbDF} F. Dahmani, V. Guirardel, recognizing a
  relatively hyperbolic group by its
  Dehn Fillings,  arXiv:1506.03233


\bibitem[DGO]{DGO} F. Dahmani, V. Guirardel, D. Osin, Hyperbolically embedded subgroups and rotating families in groups acting on hyperbolic spaces,  arXiv:1111.7048.

\bibitem[DGr]{DGr} F. Dahmani, D.  Groves,  The isomorphism problem for
  toral relatively hyperbolic groups. Publ. Math. Inst. Hautes
  \'Etudes Sci. No. 107 (2008), 211--290. 



\bibitem[D]{DCPout}  F. Dahmani, On suspensions, and conjugacy of
  hyperbolic automorphisms, Trans. Amer. Math. Soc. to appear.


\bibitem[FH]{FH14} M. Feighn, M. Handel, Algorithmic constructions of
  relative train track maps and CTs, arXiv:1411.6302 

\bibitem[GLu]{GL} F. Gautero, M. Lustig,  The mapping torus of a free group automorphism is hyperbolic relative to the canonical subgroups of polynomial growth. Preprint 2007

\bibitem[GM]{GM} D. Groves, J. Manning, Dehn filling in relatively
  hyperbolic groups, Israel J.  Math., 168 (2008), 317-429.

\bibitem[GL-10]{GL10} V. Guirardel, G. Levitt,  JSJ decompositions:
  definitions, existence, uniqueness II. Compatibility and
  acylindricity. arXiv:1002.4564.


\bibitem[GL-11]{GL_JSJ} V. Guirardel, G. Levitt,  Trees of cylinders and
canonical splittings. Geom. Topol., 15(2) (2011) 977--1012.



\bibitem[L]{Lev09} G. Levitt, Counting growth types of automorphisms
  of free groups. Geom. Funct. Anal. 19, no.4, 1119-1146 (2009). 


\bibitem[Lu-00]{Lu1}
M.~Lustig.
 Structure and conjugacy for
automorphisms of free groups I.
MPI-Bonn preprint series 2000, No.
{241}; http://www.mpim-bonn.mpg.de/preprints

\bibitem[Lu-01]{Lu2}
M.~Lustig.
 Structure and conjugacy for
automorphisms of free groups II.
MPI-Bonn preprint series 2001, No.
{4}; http://www.mpim-bonn.mpg.de/preprints



\bibitem[O-06a]{Osin} D. Osin  Relatively hyperbolic groups, Intrinsic
  geometry, algebraic properties, and algorithmic problems,
  Mem. Amer. Math. Soc. 179 (2006) no. 843. vi+100pp. 

\bibitem[O-06b]{OsinIJAC} D. Osin, Elementary subgroups of relatively
  hyperbolic groups, and bounded generation,  Internat. J. Algebra
  Comput. 16 (2006) no.1 99-118. 

\bibitem[O-07]{O_DF} D. Osin, Peripheral fillings of relatively hyperbolic groups. 
Invent. Math. 167 (2007), no. 2, 295--326.  

\bibitem[P]{Pap} P. Papasoglu, An algorithm detecting hyperbolicity. Geometric and computational perspectives on infinite groups (Minneapolis, MN and New Brunswick, NJ, 1994), 193--200, DIMACS Ser. Discrete Math. Theoret. Comput. Sci., 25, Amer. Math. Soc., Providence, RI, 1996.

\bibitem[Se]{Sel} Z. Sela,  The isomorphism problem for hyperbolic groups,
  Ann. Math., 141 (1995), 217--283.

\bibitem[Si]{Sis} A. Sisto, Contracting elements and random walks, arXiv:1112.2666.



\bibitem[T]{Tou} N. Touikan, Finding tracks in 2-complexes,  arXiv:0906.3902. 
\end{thebibliography}
\end{document}